\documentclass{article}%
\usepackage{amsmath}
\usepackage{chicago}%
\usepackage{amsfonts}%
\usepackage{amssymb}%
\usepackage{graphicx}
\newtheorem{theorem}{Theorem}

\newtheorem{definition}[theorem]{Definition}

\newtheorem{lemma}[theorem]{Lemma}

\newtheorem{remark}[theorem]{Remark}

\newenvironment{proof}[1][Proof]{\noindent\textbf{#1.} }{\ \rule{0.5em}{0.5em}}
\begin{document}

\title{Finite-approximate controllability of evolution systems via resolvent-like operators}
\author{N. I. Mahmudov\\Department of Mathematics, Eastern Mediterranean University,\\
 Gazimagusa, T.R. North Cyprus, via
Mersin 10, Turkey \\email: nazim.mahmudov@emu.edu.tr}
\date{5 Jan 2018}
\maketitle

\begin{abstract}
In this work we extend a variational method to study the approximate
controllability and finite dimensional exact controllability (
finite-approximate controllability) for the semilinear evolution equations in
Hilbert spaces. We state a useful characterization of the finite-approximate
controllability for linear evolution equation in terms of resolvent-like
operators. We also find a control so that, in addition to the approximate
controllability requirement, it ensures finite dimensional exact
controllability. Assuming the approximate controllability of the corresponding
linearized equation we obtain sufficient conditions for the finite-approximate
controllability of the semilinear evolution equation under natural conditions.
The obtained results are generalization and continuation of the recent results
on this issue. Applications to heat equations are treated.

\end{abstract}

\section{Introduction}

Controllability is one of the basic qualitative concepts in modern
mathematical control theory that play an important role in deterministic and
stochastic control theory. From mathematical point of view, exact and
approximate controllability problems should be distinguished. Exact
controllability enables to steer the system to an arbitrary final state while
approximate controllability means that the system can be steered to an
arbitrary small neighborhood of final state, and very often approximate
controllability is completely adequate in applications. If the semigroup
associated with the system is compact, the controllability operator is also
compact, and therefore the inverse fails to exist. Hence, the concept of exact
controllability is very strong and feasibility is limited; approximately
controllability is a weaker concept that is entirely appropriate in practice.
We would like to mention some interesting works: Triggiani \cite{trig1},
\cite{trig2}, Bashirov and Mahmudov \cite{mah1}, Yamamoto and Park \cite{yam},
Naito \cite{naito}, Zhou \cite{zhou1}, \cite{zhou2}, Seidman \cite{seid}, Li
and Yong \cite{yong2}, Mahmudov \cite{mah2}, \cite{mah4}. Also, there are many
papers on the approximate controllability of the various types of nonlinear
systems under different conditions (see \cite{mah1}-\cite{yong2} and
references therein).

In this paper we will study a stronger version of controllability concept that
is referred to as the finite-approximate controllability problem. It should be
stressed out that in the context of abstract linear control systems,
finite-approximate controllability problem is a consequence of approximate
one, see \cite{lions}. So these two concepts are equivalent. However in the
nonlinear context they are not equivalent, see \cite{zuazua1}. Recently
finite-approximate controllability result for abstract semilinear evolution
equations with compact $C_{0}$-semigroup is presented in \cite{mah3}.

In this paper, we investigate simultaneous approximate and finite-dimensional
exact controllability (finite-approximate controllability) of the following
semilinear evolution system :%

\begin{equation}
\left\{
\begin{array}
[c]{c}%
y^{\prime}\left(  t\right)  =Ay\left(  t\right)  +Bu\left(  t\right)
+f\left(  t,y\left(  t\right)  \right)  +g\left(  t,y\left(  t\right)
\right)  ,\ \ t\in\left[  0,T\right]  ,\\
y\left(  0\right)  =y_{0},
\end{array}
\right.  \label{mp1}%
\end{equation}
where the state variable $y\left(  \cdot\right)  $ takes values in the Hilbert
space $\mathfrak{X}$, $A:D\left(  A\right)  \subset\mathfrak{X}\rightarrow
\mathfrak{X}$ is a family of closed and bounded linear operators generating a
strongly continuous semigroup $\mathfrak{U}:\left[  0,T\right]  \rightarrow
L\left(  \mathfrak{X}\right)  $, where the domain $D\left(  A\right)
\subset\mathfrak{X}$ which is dense in $\mathfrak{X},$ the control function
$u\left(  \cdot\right)  $ is given in $L^{2}\left(  \left[  0,T\right]
,U\right)  ,$ $U$ is a Hilbert space, $B$ is a bounded linear operator from
$U$ into $\mathfrak{X},$ $f,g:\left[  0,T\right]  \times\mathfrak{X}%
\rightarrow\mathfrak{X}$ are given functions satisfying some assumptions
specified later and $y_{0}$ is an element of the Hilbert space $\mathfrak{X}$.

We present the following definition of mild solutions of system (\ref{mp1}).

\begin{definition}
\label{def:mildsol} $y\in C\left(  \left[  0,T\right]  ,\mathfrak{X}\right)  $
is called a mild solution of (\ref{mp1}) if%
\begin{equation}
y\left(  t\right)  =\mathfrak{U}\left(  t\right)  x_{0}+\int_{0}%
^{t}\mathfrak{U}\left(  t-s\right)  \left[  Bu\left(  s\right)  +f\left(
s,y\left(  s\right)  \right)  +g\left(  s,y\left(  s\right)  \right)  \right]
ds,\ \ t\in\left[  0,T\right]  . \label{mild}%
\end{equation}

\end{definition}

Following \cite{curtain}, we define the controllability concepts and
controllability operator for the system (\ref{mp1}).

\begin{definition}
\label{def:1}For the system (\ref{mp1}), we define the following concepts:
\newline(a) A controllability operator is the bounded linear operator
$B_{0}^{T}:L^{2}\left(  \left[  0,T\right]  ,U\right)  $\newline%
$\rightarrow\mathfrak{X}$ defined by%
\[
L_{0}^{T}u:=\int_{0}^{T}\mathfrak{U}\left(  T-s\right)  Bu\left(  s\right)
ds;
\]
\newline (b) Control system (\ref{mp1}) is approximately controllable on
$\left[  0,T\right]  ,$ if for every $y_{0},y_{f}\in\mathfrak{X}$, and for
every $\varepsilon>0$, there exists a control $u\in L^{2}\left(  \left[
0,T\right]  ,U\right)  $ such that the mild solution $y$ of the Cauchy problem
(\ref{mp1}) satisfies $y\left(  0\right)  =y_{0}$ and $\left\Vert y\left(
T\right)  -y_{f}\right\Vert <\varepsilon$. \newline(c)\ Let $M$ be a finite
dimensional subspace of $\mathfrak{X}$ and let us denote by $\pi_{M}$ the
orthogonal projection from $\mathfrak{X}$ into $M$. Control system (\ref{mp1})
is finite-approximately controllable on $\left[  0,T\right]  ,$ if for every
$y_{0},y_{f}\in\mathfrak{X}$, and for every $\varepsilon>0$, there exists a
control $u\in L^{2}\left(  \left[  0,T\right]  ,U\right)  $ such that the mild
solution $y$ of the Cauchy problem (\ref{mp1}) satisfies $y\left(  0\right)
=y_{0}$ and $\left\Vert y\left(  T\right)  -y_{f}\right\Vert <\varepsilon
\ $and $\pi_{M}y\left(  T\right)  =\pi_{M}y_{f}$. \newline(d) The
controllability Gramian is defined by%
\begin{equation}
\Gamma_{0}^{T}:=L_{0}^{T}\left(  L_{0}^{T}\right)  ^{\ast}=\int_{0}%
^{T}\mathfrak{U}\left(  T-s\right)  BB^{\ast}\mathfrak{U}^{\ast}\left(
T-s\right)  ds:\mathfrak{X}\rightarrow\mathfrak{X}.\nonumber
\end{equation}

\end{definition}

The rest of this paper is organized as follows. In section 2, we will present
some results on properties of positive linear compact operators depending on
parameter. We define a resolvent-like operators and give necessary and
sufficient conditions for finite-approximate controllability of linear
evolution equations. Section 3 is divided into two subsections. In subsection
3.1, using a control defined by resolvent-like operator we define a control
operator $\Theta_{\varepsilon}$ and show existence of fixed points. In
subsection 3.2 we prove our main result on finite-approximate controllability
of semilinear evolution system. Finally, we present two examples to
demonstrate our main results in section 4.

Several comments are in order:

(i) The variational approach developed in this paper is somewhat different
from those applied in the literature and provide another new method to prove
simultaneous approximate and exact finite-dimensional controllability for
(\ref{mp1}).

(ii) The proof of the main result obtained in this paper are based on quasi
linearization of semilinear problem and on viewing the finite-approximate
controllability problem as a limit of optimal control problems. It combines
the methods used in papers \cite{zuazua1}, \cite{mah4} and \cite{mah2}.

(iii) Requirement of exactly controlling the finite-dimensional projection
introduces new difficulties. To overcome it, we present criteria for
finite-approximate controllability of linear systems in terms of
resolvent-like operators, and study convergence properties of approximating
resolvent-like operators.

(iv) The variational approach developed here is constructive since
approximating control can be given explicitly. It is interesting both from the
theoretical and the numerical point of view.

(v) One may expect the results of this paper to hold for a class of problems
governed by different type of evolution systems such as Caputo fractional
differential equations (FDEs), Riemann-Liouville FDEs, stochastic FDEs,
Sobolev type FDEs and so on.

\section{Finite-approximate controllability of linear systems}

In the present section we investigate the finite-approximate controllability
of linear evolution system:%
\begin{equation}
\left\{
\begin{array}
[c]{c}%
y^{\prime}\left(  t\right)  =Ay\left(  t\right)  +Bu\left(  t\right)
,\ \ \ t\in\left[  0,T\right]  ,\\
y\left(  0\right)  =y_{0}.
\end{array}
\right.  \label{lp1}%
\end{equation}
Finite-approximate controllability concept was introduced in \cite{zuazua1}.
This property not only says that the distance between $y\left(  T\right)  $
and the target $y_{f}$ is small but also that the projections of $y\left(
T\right)  $ and $y_{f}$ over $M$ coincide.

It is known that the resolvent operator $\left(  \varepsilon I+\Gamma_{0}%
^{T}\right)  ^{-1}$ is useful in studying the controllability properties of
linear and semilinear systems, see \cite{mah1}, \cite{mah2}. In this respect,
we state a useful characterization of the finite-approximate controllability
for (\ref{lp1}) in terms of resolvent-like operator. We show that for the
linear evolution system (\ref{lp1}) approximate controllability on $\left[
0,T\right]  $ is equivalent to the finite-approximate controllability on
$\left[  0,T\right]  .$ Moreover, we present necessary and sufficient
conditions for the finite-approximate controllability of linear evolution
systems in Hilbert spaces in terms of resolvent-like operators.

Firstly, we present three results on the resolvent operators.

\begin{theorem}
\label{thm:1}Assume that $\Gamma\left(  \varepsilon\right)  ,\Gamma
:\mathfrak{X}\rightarrow\mathfrak{X},$ $\varepsilon>0$, are linear positive
operators such that
\[
\lim_{\varepsilon\rightarrow0^{+}}\left\Vert \Gamma\left(  \varepsilon\right)
h-\Gamma h\right\Vert =0,\ h\in\mathfrak{X}.
\]
Then for any sequence $\left\{  \varepsilon_{n}>0\right\}  $ converging to $0$
as $n\rightarrow\infty$, we have
\[
\lim_{n\rightarrow\infty}\left\Vert \varepsilon_{n}\left(  \varepsilon
_{n}I+\Gamma\left(  \varepsilon_{n}\right)  \right)  ^{-1}\pi_{M}\right\Vert
=0.
\]

\end{theorem}

\begin{proof}
It is clear that $\left(  \varepsilon I+\Gamma\left(  \varepsilon\right)
\right)  ^{-1}\pi_{M}$ maps $\mathfrak{X}$ into finite dimensional space
$\operatorname{Im}\left(  \left(  \varepsilon I+\Gamma\left(  \varepsilon
\right)  \right)  ^{-1}\pi_{M}\right)  $ and%
\[
0\leq\left\Vert \varepsilon\left(  \varepsilon I+\Gamma\left(  \varepsilon
\right)  \right)  ^{-1}\pi_{M}\right\Vert \leq1.
\]
Then for any sequence $\left\{  \varepsilon_{n}>0\right\}  $ converging to $0$
as $n\rightarrow\infty,$ we have
\[
0\leq\rho:=\lim_{n\rightarrow\infty}\left\Vert \varepsilon_{n}\left(
\varepsilon_{n}I+\Gamma\left(  \varepsilon_{n}\right)  \right)  ^{-1}\pi
_{M}\right\Vert \leq1.
\]
Show that $\rho=0.$ Let $\left\Vert \varepsilon_{n}\left(  \varepsilon
_{n}I+\Gamma\left(  \varepsilon_{n}\right)  \right)  ^{-1}\pi_{M}\right\Vert
:=\gamma_{n}$. $\ $Then $0\leq\lim_{n\rightarrow\infty}\gamma_{n}=\rho\leq1$
and by the definition of $\gamma_{n}$ there exists a sequence $\left\{
h_{n,m}\in\mathfrak{X}:\left\Vert h_{n,m}\right\Vert =1\right\}  $ such that
\begin{align*}
\varepsilon_{n}\left(  \varepsilon_{n}I+\Gamma\left(  \varepsilon_{n}\right)
\right)  ^{-1}\pi_{M}h_{n,m}  &  =:z_{n,m},\ \ \ \ \\
0  &  \leq\left\Vert z_{n,m}\right\Vert \leq1,\ \ \ \left\Vert z_{n,m}%
\right\Vert \rightarrow\gamma_{n}\ \ \text{as\ \ }m\rightarrow\infty.
\end{align*}
It follows that%
\begin{equation}
\varepsilon_{n}\pi_{M}h_{n,m}=\varepsilon_{n}z_{n,m}+\Gamma\left(
\varepsilon_{n}\right)  z_{n,m}. \label{f121}%
\end{equation}
Since $\left\{  \pi_{M}h_{n,m}\right\}  $ and $\left\{  z_{n,m}\right\}  $ are
bounded sequences of finite dimensional vectors, without loss of generality we
may assume that%
\[
z_{n,m}\rightarrow z_{n}\ \ \ \text{and\ }\pi_{M}h_{n,m}\rightarrow
h_{n}\ \text{strongly as }m\rightarrow\infty.
\]
Taking limit as $m\rightarrow\infty$ in (\ref{f121}), we get%
\begin{equation}
\varepsilon_{n}h_{n}=\varepsilon_{n}z_{n}+\Gamma\left(  \varepsilon
_{n}\right)  z_{n},\ \ \ \ \ \left\Vert h_{n}\right\Vert \leq1,\ \ 0\leq
\left\Vert z_{n}\right\Vert =\gamma_{n}\leq1. \label{ff11}%
\end{equation}
Next, having in mind that $z_{n}\rightarrow z$ along some subsequence, we take
limit as $n\rightarrow\infty$ $\left(  \ref{ff11}\right)  $ to get
\begin{gather*}
0=\lim_{n\rightarrow\infty}\Gamma\left(  \varepsilon_{n}\right)  z_{n}%
=\lim_{n\rightarrow\infty}\left(  \Gamma\left(  \varepsilon_{n}\right)
-\Gamma\right)  z+\lim_{n\rightarrow\infty}\Gamma\left(  \varepsilon
_{n}\right)  \left(  z_{n}-z\right)  +\Gamma z=\Gamma z=0,\\
\Gamma z=0\Longrightarrow z=0.
\end{gather*}
By definition of the positive operator $\Gamma z=0$ implies that $z=0$. Thus
\[
\rho=\lim_{n\rightarrow\infty}\left\Vert \varepsilon_{n}\left(  \varepsilon
_{n}I+\Gamma\left(  \varepsilon_{n}\right)  \right)  ^{-1}\pi_{M}\right\Vert
=\lim_{n\rightarrow\infty}\gamma_{n}=\lim_{n\rightarrow\infty}\left\Vert
z_{n}\right\Vert =\left\Vert z\right\Vert =0.
\]
The theorem is proved.
\end{proof}

\begin{theorem}
\label{thm:2}Assume that $\Gamma\left(  \varepsilon\right)  :\mathfrak{X}%
\rightarrow\mathfrak{X},$ $\varepsilon>0$, are linear positive operators. Then
for any $\varepsilon>0$ we have $\left\Vert \varepsilon\left(  \varepsilon
I+\Gamma\left(  \varepsilon\right)  \right)  ^{-1}\pi_{M}\right\Vert <1.$
\end{theorem}

\begin{proof}
It is clear that $\left(  \varepsilon I+\Gamma\left(  \varepsilon\right)
\right)  ^{-1}\pi_{M}$ maps $\mathfrak{X}$ into finite dimensional subspace of
$\mathfrak{X}$ and
\[
\left\Vert \varepsilon\left(  \varepsilon I+\Gamma\left(  \varepsilon\right)
\right)  ^{-1}\pi_{M}\right\Vert \leq1.
\]
Let us show that $\left\Vert \varepsilon\left(  \varepsilon I+\Gamma\left(
\varepsilon\right)  \right)  ^{-1}\pi_{M}\right\Vert <1.$ Contrary, assume
that there exists a sequence $\left\{  h_{n}\in\mathfrak{X}:\left\Vert
h_{n}\right\Vert =1\right\}  $ such that
\begin{equation}
\varepsilon\left(  \varepsilon I+\Gamma\left(  \varepsilon\right)  \right)
^{-1}\pi_{M}h_{n}=:z_{n},\ \ \ \ \left\Vert z_{n}\right\Vert \rightarrow
1\ \ \text{as\ \ }n\rightarrow\infty. \label{f0}%
\end{equation}
It follows that $\left\{  z_{n}\right\}  $ is a sequence of finite dimensional
vectors and
\begin{equation}
\varepsilon\pi_{M}h_{n}=\varepsilon z_{n}+\Gamma\left(  \varepsilon\right)
z_{n}\ \ \text{and}\ \ z_{n}\rightarrow z_{0}\ \ \text{strongly\ in}%
\ \ \mathfrak{X}. \label{f1}%
\end{equation}%
\begin{align*}
\left\langle \pi_{M}h_{n},z_{n}\right\rangle  &  =\left\langle z_{n}%
,z_{n}\right\rangle +\frac{1}{\varepsilon}\left\langle \Gamma\left(
\varepsilon\right)  z_{n},z_{n}\right\rangle ,\\
\left\Vert z_{n}\right\Vert ^{2}  &  <\left\langle z_{n},z_{n}\right\rangle
+\frac{1}{\varepsilon}\left\langle \Gamma\left(  \varepsilon\right)
z_{n},z_{n}\right\rangle =\left\langle \pi_{M}h_{n},z_{n}\right\rangle
\leq\left\Vert \pi_{M}h_{n}\right\Vert \left\Vert z_{n}\right\Vert
\leq\left\Vert z_{n}\right\Vert .\
\end{align*}
Taking limit as $n\rightarrow\infty$ we get%
\begin{align*}
1  &  \leq1+\frac{1}{\varepsilon}\left\langle \Gamma\left(  \varepsilon
\right)  z_{0},z_{0}\right\rangle \leq1,\\
\left\langle \Gamma\left(  \varepsilon\right)  z_{0},z_{0}\right\rangle  &
=0\Longrightarrow z_{0}=0.
\end{align*}
Now from (\ref{f1}) it follows that $\left\Vert z_{n}\right\Vert
\rightarrow0\ \ $as\ \ $n\rightarrow\infty.$ Contradiction.
\end{proof}

\begin{theorem}
If $\Gamma:\mathfrak{X}\rightarrow\mathfrak{X}$ is a linear nonnegative
operator then the operator $\varepsilon\left(  I-\pi_{M}\right)
+\Gamma:\mathfrak{X\rightarrow X}$ is invertible and%
\begin{equation}
\left\Vert \left(  \varepsilon\left(  I-\pi_{M}\right)  +\Gamma\right)
^{-1}h\right\Vert \leq\frac{1}{\min\left(  \varepsilon,\delta\right)
}\left\Vert h\right\Vert ,\ \ h\in\mathfrak{X,} \label{ww1}%
\end{equation}
where $\delta=\min\left\{  \left\langle \pi_{M}\Gamma\pi_{M}\varphi
,\varphi\right\rangle :\left\Vert \pi_{M}\varphi\right\Vert =1\right\}  $.
Moreover, if $\Gamma:\mathfrak{X}\rightarrow\mathfrak{X}$ is a linear positive
operator then
\begin{equation}
\left(  \varepsilon\left(  I-\pi_{M}\right)  +\Gamma\right)  ^{-1}=\left(
I-\varepsilon\left(  \varepsilon I+\Gamma\right)  ^{-1}\pi_{M}\right)
^{-1}\left(  \varepsilon I+\Gamma\right)  ^{-1}. \label{f2}%
\end{equation}

\end{theorem}

\begin{proof}
We write $\varepsilon\left(  I-\pi_{M}\right)  +\Gamma$ as follows.%
\[
\varepsilon\left(  I-\pi_{M}\right)  +\Gamma=\varepsilon\left(  I-\pi
_{M}\right)  +\left(  I-\pi_{M}\right)  \Gamma+\pi_{M}\Gamma.
\]
It is clear that%
\begin{align*}
&  \left\langle \left(  \varepsilon\left(  I-\pi_{M}\right)  +\Gamma\right)
\varphi,\varphi\right\rangle \\
&  =\left\langle \left(  \varepsilon\left(  I-\pi_{M}\right)  +\left(
I-\pi_{M}\right)  \Gamma\right)  \varphi,\varphi\right\rangle +\left\langle
\pi_{M}\Gamma\varphi,\varphi\right\rangle \\
&  \geq\left\{
\begin{tabular}
[c]{ll}%
$\left\langle \pi_{M}\Gamma\pi_{M}\varphi,\varphi\right\rangle ,$ &
$\varphi\in M,$\\
$\left\langle \varepsilon\left(  I-\pi_{M}\right)  \varphi+\left(  I-\pi
_{M}\right)  \Gamma\left(  I-\pi_{M}\right)  \varphi,\varphi\right\rangle ,$ &
$\varphi\in\mathfrak{X}\ominus M$%
\end{tabular}
\ \right. \\
&  \geq\min\left(  \varepsilon,\delta\right)  \left\Vert \varphi\right\Vert
^{2}.
\end{align*}
It follows that $\varepsilon\left(  I-\pi_{M}\right)  +\Gamma$ is invertible
and (\ref{ww1}) is satisfied.

If $\Gamma:\mathfrak{X}\rightarrow\mathfrak{X}$ is a linear positive operator
then by Theorem \ref{thm:2}, $\left(  I-\varepsilon\left(  \varepsilon
I+\Gamma\right)  ^{-1}\pi_{M}\right)  ^{-1}$ exists. On the other hand, since
$\left(  \varepsilon I+\Gamma\right)  $ is invertible and%
\[
\varepsilon\left(  I-\pi_{M}\right)  +\Gamma=\left(  \varepsilon
I+\Gamma\right)  \left(  I-\varepsilon\left(  \varepsilon I+\Gamma\right)
^{-1}\pi_{M}\right)  ,
\]
the operator $\varepsilon\left(  I-\pi_{M}\right)  +\Gamma$ is boundedly
invertible and (\ref{f2}) is satisfied.
\end{proof}

Next, we present new criteria for the finite-approximate controllability of
linear evolution equations.

\begin{theorem}
\label{thm:4}The following statements are equivalent:\newline (i) the system
(\ref{lp1}) is approximately controllable on $\left[  0,T\right]  ;$%
\newline(ii) $\Gamma_{0}^{T}$ is positive, that is $\left\langle \Gamma
_{0}^{T}x,x\right\rangle >0$ for all $0\neq x\in\mathfrak{X}$;\newline(iii)
$\varepsilon\left(  \varepsilon I+\Gamma_{0}^{T}\right)  ^{-1}\rightarrow0$ as
$\varepsilon\rightarrow0^{+}$ in the strong operator topology;\newline(iv)
$\varepsilon\left(  \varepsilon\left(  I-\pi_{M}\right)  +\Gamma_{0}%
^{T}\right)  ^{-1}\rightarrow0$ as $\varepsilon\rightarrow0^{+}$ in the strong
operator topology;\newline(v) the system (\ref{lp1}) is finite-approximately
controllable on $\left[  0,T\right]  $.
\end{theorem}

\begin{proof}
The equivalences (i)$\Leftrightarrow$(ii)$\Leftrightarrow$(iii) are well
known, see \cite{mah1}.

For the equivalence (iii)$\Longleftrightarrow$(v), for any $\varepsilon>0,$
$h\in\mathfrak{X}$, consider the following functional $J_{\varepsilon}\left(
\cdot,h\right)  :\mathfrak{X}\rightarrow R:$%
\[
J_{\varepsilon}\left(  \varphi,h\right)  =\frac{1}{2}\int_{0}^{T}\left\Vert
B^{\ast}\mathfrak{U}^{\ast}\left(  T-s\right)  \varphi\right\Vert ^{2}%
ds+\frac{\varepsilon}{2}\left\langle \left(  I-\pi_{M}\right)  \varphi
,\varphi\right\rangle -\left\langle \varphi,h-\mathfrak{U}\left(  T\right)
x_{0}\right\rangle .
\]
Assume that (iii)($\Longleftrightarrow$(ii)) is satisfied. It is clear that
$J_{\varepsilon}\left(  \cdot,h\right)  $ is Gateaux differentiable,
$J_{\varepsilon}^{\prime}\left(  \varphi,h\right)  =\Gamma_{0}^{T}%
\varphi+\varepsilon\left(  I-\pi_{M}\right)  \varphi-h+\mathfrak{U}\left(
T\right)  x_{0}$ is strictly monotonic and consequently $J_{\varepsilon
}\left(  \cdot,h\right)  $ is strictly convex, since $\Gamma_{0}^{T}$ is
positive. Thus $J_{\varepsilon}\left(  \cdot,h\right)  $ has a unique minimum
and can be found as follows:
\begin{align*}
\Gamma_{0}^{T}\varphi+\varepsilon\left(  I-\pi_{M}\right)  \varphi
-h+\mathfrak{U}\left(  T\right)  x_{0}  &  =0,\\
\varphi_{\min}  &  =-\left(  \varepsilon\left(  I-\pi_{M}\right)  +\Gamma
_{0}^{T}\right)  ^{-1}\left(  \mathfrak{U}\left(  T\right)  x_{0}-h\right)  .
\end{align*}
It follows that for the control $u_{\varepsilon}\left(  s\right)  =B^{\ast
}\mathfrak{U}^{\ast}\left(  T-s\right)  \varphi_{\min}$%
\begin{align}
x_{\varepsilon}\left(  T\right)  -h  &  =\mathfrak{U}\left(  T\right)
x_{0}+\int_{0}^{T}\mathfrak{U}\left(  T-s\right)  Bu\left(  s\right)
ds-h\nonumber\\
&  =\mathfrak{U}\left(  T\right)  x_{0}-h-\Gamma_{0}^{T}\left(  \varepsilon
\left(  I-\pi_{M}\right)  +\Gamma_{0}^{T}\right)  ^{-1}\left(  \mathfrak{U}%
\left(  T\right)  x_{0}-h\right) \nonumber\\
&  =\mathfrak{U}\left(  T\right)  x_{0}-h-\left(  \Gamma_{0}^{T}%
+\varepsilon\left(  I-\pi_{M}\right)  -\varepsilon\left(  I-\pi_{M}\right)
\right) \nonumber\\
&  \times\left(  \varepsilon\left(  I-\pi_{M}\right)  +\Gamma_{0}^{T}\right)
^{-1}\left(  \mathfrak{U}\left(  T\right)  x_{0}-h\right) \nonumber\\
&  =\varepsilon\left(  I-\pi_{M}\right)  \left(  \varepsilon\left(  I-\pi
_{M}\right)  +\Gamma_{0}^{T}\right)  ^{-1}\left(  \mathfrak{U}\left(
T\right)  x_{0}-h\right)  . \label{f3}%
\end{align}
Thus%
\begin{align*}
\lim_{\varepsilon\rightarrow0^{+}}\left\Vert x_{\varepsilon}\left(  T\right)
-h\right\Vert  &  =\lim_{\varepsilon\rightarrow0^{+}}\varepsilon\left\Vert
\left(  I-\pi_{M}\right)  \left(  \varepsilon\left(  I-\pi_{M}\right)
+\Gamma_{0}^{T}\right)  ^{-1}\left(  \mathfrak{U}\left(  T\right)
x_{0}-h\right)  \right\Vert =0,\\
\pi_{M}\left(  x_{\varepsilon}\left(  T\right)  -h\right)   &  =0,
\end{align*}
that is the system (\ref{lp1}) is finite-approximately controllable on
$\left[  0,T\right]  .$ Thus (iii)$\Longrightarrow$(v). The implication
(v)$\Rightarrow$(iii) is obvious, since finite-approximate controllability
implies the approximate controllability.

For the implication (iii)$\Rightarrow$(iv), suppose that for any
$h\in\mathfrak{X}$%
\[
\lim_{\varepsilon\rightarrow0^{+}}\left\Vert \left(  \varepsilon I+\Gamma
_{0}^{T}\right)  ^{-1}h\right\Vert =0.
\]
From (\ref{f2}) it follows that for any $h\in\mathfrak{X}$%
\begin{align}
\left\Vert \varepsilon\left(  \varepsilon\left(  I-\pi_{M}\right)  +\Gamma
_{0}^{T}\right)  ^{-1}h\right\Vert  &  \leq\left\Vert \left(  I-\varepsilon
\left(  \varepsilon I+\Gamma_{0}^{T}\right)  ^{-1}\pi_{M}\right)
^{-1}\right\Vert \left\Vert \varepsilon\left(  \varepsilon I+\Gamma_{0}%
^{T}\right)  ^{-1}h\right\Vert \nonumber\\
&  \leq\dfrac{1}{1-\left\Vert \varepsilon\left(  \varepsilon I+\Gamma_{0}%
^{T}\right)  ^{-1}\pi_{M}\right\Vert }\left\Vert \varepsilon\left(
\varepsilon I+\Gamma_{0}^{T}\right)  ^{-1}h\right\Vert . \label{f5}%
\end{align}
On the other hand, from%
\begin{align*}
&  \varepsilon_{1}\left(  \varepsilon_{1}I+\Gamma\right)  ^{-1}\pi
_{M}-\varepsilon\left(  \varepsilon I+\Gamma\right)  ^{-1}\pi_{M}\\
&  =\varepsilon_{1}\left(  \varepsilon_{1}I+\Gamma\right)  ^{-1}\left(
I+\varepsilon^{-1}\Gamma-I-\varepsilon_{1}^{-1}\Gamma\right)  \varepsilon
\left(  \varepsilon I+\Gamma\right)  ^{-1}\pi_{M}\\
&  =\varepsilon_{1}\left(  \varepsilon_{1}I+\Gamma\right)  ^{-1}\left(
\varepsilon^{-1}\Gamma-\varepsilon_{1}^{-1}\Gamma\right)  \varepsilon\left(
\varepsilon I+\Gamma\right)  ^{-1}\pi_{M}\\
&  =\left(  \varepsilon_{1}I+\Gamma\right)  ^{-1}\left(  \varepsilon_{1}%
\Gamma-\varepsilon\Gamma\right)  \left(  \varepsilon I+\Gamma\right)  ^{-1}%
\pi_{M}\\
&  =\left(  \varepsilon_{1}I+\Gamma\right)  ^{-1}\left(  \varepsilon
_{1}-\varepsilon\right)  \Gamma\left(  \varepsilon I+\Gamma\right)  ^{-1}%
\pi_{M},
\end{align*}
it follows that $\varepsilon\left(  \varepsilon I+\Gamma\right)  ^{-1}\pi_{M}$
is continuous in $\varepsilon.$ Indeed,%
\[
\left\Vert \varepsilon_{1}\left(  \varepsilon_{1}I+\Gamma\right)  ^{-1}\pi
_{M}-\varepsilon\left(  \varepsilon I+\Gamma\right)  ^{-1}\pi_{M}\right\Vert
\leq\dfrac{\left\vert \varepsilon_{1}-\varepsilon\right\vert }{\varepsilon
_{1}}\rightarrow0\ \ \ \text{as\ \ }\varepsilon_{1}\rightarrow\varepsilon.
\]
By (\ref{f5}), continuity of $\varepsilon\left(  \varepsilon I+\Gamma\right)
^{-1}\pi_{M}$ and Theorem \ref{thm:2}, we have
\begin{align*}
\gamma &  =\max_{0\leq\varepsilon\leq1}\left\Vert \varepsilon\left(
\varepsilon I+\Gamma_{0}^{T}\right)  ^{-1}\pi_{M}\right\Vert <1,\\
\left\Vert \varepsilon\left(  \varepsilon\left(  I-\pi_{M}\right)  +\Gamma
_{0}^{T}\right)  ^{-1}h\right\Vert  &  \leq\dfrac{1}{1-\gamma}\left\Vert
\varepsilon\left(  \varepsilon I+\Gamma_{0}^{T}\right)  ^{-1}h\right\Vert .
\end{align*}
Thus $\varepsilon\left(  \varepsilon\left(  I-\pi_{M}\right)  +\Gamma_{0}%
^{T}\right)  ^{-1}$ converges to zero as $\varepsilon\rightarrow0^{+}$ in the
strong operator topology.

The implication (iv)$\Rightarrow$(v) follows from (\ref{f3}).
\end{proof}

\begin{remark}
Analogue of Theorem \ref{thm:4} is true for different kind of equations such
as fractional linear differential equations with Caputo derivative, fractional
linear differential equations with Riemann-Liouville derivative, Fredholm type
linear integral equations and so on.
\end{remark}

\section{Finite-approximate controllability of semilinear system}

In this section, we first show that for every $\varepsilon>0$ and every final
state $y_{f}\in\mathfrak{X}$, the integral equation
\[
z\left(  t\right)  =\mathfrak{T}\left(  t,0;F\left(  z\right)  \right)
x_{0}+\int_{0}^{t}\mathfrak{T}\left(  t,s;F\left(  z\right)  \right)  \left[
Bu_{\varepsilon}\left(  s,z\right)  +g\left(  s,z\left(  s\right)  \right)
\right]  ds,
\]
with the control
\begin{align*}
u_{\varepsilon}\left(  t,z\right)   &  =B^{\ast}\mathfrak{T}^{\ast}\left(
T,t;F\left(  z\right)  \right)  \left(  \varepsilon\left(  I-\pi_{M}\right)
+\Gamma_{0}^{T}\left(  F\left(  z\right)  \right)  \right)  ^{-1}\\
&  \times\left(  h-\mathfrak{T}\left(  T,0;F\left(  z\right)  \right)
x_{0}-\int_{0}^{T}\mathfrak{T}\left(  T,s;F\left(  z\right)  \right)  g\left(
s,z\left(  s\right)  \right)  ds\right)
\end{align*}
has at least one solution, say $y_{\varepsilon}^{\ast}$. Then we can
approximate any point $y_{f}\in\mathfrak{X}$ by using these solutions
$y_{\varepsilon}^{\ast}$, $\varepsilon>0$.

\subsection{Existence of fixed point}

We impose the following assumptions:\newline

(S) $\mathfrak{X}$ and $U$ are separable Hilbert spaces, $\mathfrak{U}\left(
t\right)  ,t>0$ is a compact semigroup on $\mathfrak{X}$ and $B\in L\left(
U,\mathfrak{X}\right)  .$\newline

(F) $f:\left[  0,T\right]  \times\mathfrak{X}\rightarrow\mathfrak{X}$ is
continuous and has continuous uniformly bounded Frechet derivative
$f_{z}^{\prime}\left(  \cdot,\cdot\right)  $, that is, for some $L>0,$%
\[
\left\Vert f_{z}^{\prime}\left(  t,z\right)  \right\Vert _{L\left(
\mathfrak{X}\right)  }\leq L,\ \ \ \forall\left(  t,z\right)  \in\left[
0,T\right]  \times\mathfrak{X}.
\]
\newline

(G) $g:\left[  0,T\right]  \times\mathfrak{X}\rightarrow\mathfrak{X}$ is
continuous and there exists $m\in C\left(  \left[  0,T\right]  ,R^{+}\right)
$ such that%
\[
\left\Vert g\left(  t,z\right)  \right\Vert \leq m\left(  t\right)
,\ \ \ \forall\left(  t,z\right)  \in\left[  0,T\right]  \times\mathfrak{X}.
\]
\newline

(AC) System%
\begin{equation}
y\left(  t\right)  =\mathfrak{U}\left(  t\right)  y_{0}+\int_{0}%
^{t}\mathfrak{U}\left(  t-s\right)  \left[  Bu\left(  s\right)  +G\left(
s\right)  y\left(  s\right)  \right]  ds\label{m1}%
\end{equation}
is approximately controllable for any $G\in L^{2}\left(  0,T;L\left(
\mathfrak{X}\right)  \right)  $.

It is clear that under the conditions (S), (F) and (G), for any $y_{0}%
\in\mathfrak{X}$ and $u\left(  \cdot\right)  \in L_{2}\left(  0,T;U\right)  ,$
the system (\ref{mild}) admits a unique solution $y\left(  \cdot\right)
=y\left(  \cdot,y_{0},u\right)  .$

Define%
\begin{equation}
F\left(  t,z\right)  =\int_{0}^{1}f^{\prime}\left(  t,rz\right)
dr,\ \ \ z\in\mathfrak{X}. \label{p2}%
\end{equation}
Thanks to the assumption (F) there exists a constant $L>0$ such that operator
$F$ defined by (\ref{p2}) has the following properties:%
\begin{align*}
F  &  :\left[  0,T\right]  \times\mathfrak{X}\rightarrow L\left(
\mathfrak{X}\right)  ,\\
f\left(  t,z\right)   &  =F\left(  t,z\right)  z+f\left(  t,0\right)  ,\\
\left\Vert F\left(  t,z\left(  t\right)  \right)  \right\Vert _{L\left(
\mathfrak{X}\right)  }  &  \leq L,\ \ \ z\left(  \cdot\right)  \in C\left(
\left[  0,T\right]  ,\mathfrak{X}\right)  ,\ t\in\left[  0,T\right]  ,\\
F\left(  \cdot,\cdot\right)   &  \in C\left(  \left[  0,T\right]
\times\mathfrak{X},L\left(  \mathfrak{X}\right)  \right)  .
\end{align*}
For simplicity we assume that $f\left(  t,0\right)  \equiv0$. Then the system
(\ref{mild}) can be rewritten as follows%
\[
y\left(  t\right)  =\mathfrak{U}\left(  t\right)  y_{0}+\int_{0}%
^{t}\mathfrak{U}\left(  t-s\right)  \left[  Bu\left(  s\right)  +F\left(
s,y\left(  s\right)  \right)  y\left(  s\right)  +g\left(  s,y\left(
s\right)  \right)  \right]  ds.
\]
For any fixed $z\left(  \cdot\right)  \in C\left(  \left[  0,T\right]
,\mathfrak{X}\right)  ,$ let $y\left(  \cdot\right)  =y\left(  \cdot
,y_{0},z,u\right)  $ be the solution of%
\begin{equation}
y\left(  t\right)  =\mathfrak{U}\left(  t\right)  y_{0}+\int_{0}%
^{t}\mathfrak{U}\left(  t-s\right)  \left[  Bu\left(  s\right)  +F\left(
s,z\left(  s\right)  \right)  y\left(  s\right)  +g\left(  s,z\left(
s\right)  \right)  \right]  ds \label{p3}%
\end{equation}
or%
\[
y\left(  t\right)  =\mathfrak{T}\left(  t,0;F\left(  z\right)  \right)
y_{0}+\int_{0}^{t}\mathfrak{T}\left(  t,s;F\left(  z\right)  \right)  \left[
Bu\left(  s\right)  +g\left(  s,z\left(  s\right)  \right)  \right]  ds,
\]
where%
\[
\mathfrak{T}\left(  t,s;F\left(  z\right)  \right)  y=\mathfrak{U}\left(
t-s\right)  y+\int_{s}^{t}\mathfrak{U}\left(  t-r\right)  F\left(  r,z\left(
r\right)  \right)  \mathfrak{T}\left(  r,s;F\left(  z\right)  \right)
ydr,0\leq s\leq t\leq T.
\]
Under the above conditions we are going to show the following:

(i) For any function $z\left(  \cdot\right)  \in C\left(
\left[  0,T\right]  ,\mathfrak{X}\right)  $, there exists a control
$u_{\varepsilon}\left(  t,z\right)  $ determined explicitly by $z\left(
\cdot\right)  $, such that
\begin{align*}
u_{\varepsilon}\left(  t,z\right)   &  =B^{\ast}\mathfrak{T}^{\ast}\left(
T,t;F\left(  z\right)  \right)  \left(  \varepsilon\left(  I-\pi_{M}\right)
+\Gamma_{0}^{T}\left(  F\left(  z\right)  \right)  \right)  ^{-1}\\
&  \times\left(  h-\mathfrak{T}\left(  T,0;F\left(  z\right)  \right)
x_{0}-\int_{0}^{T}\mathfrak{T}\left(  T,s;F\left(  z\right)  \right)  g\left(
s,z\left(  s\right)  \right)  ds\right)  .
\end{align*}

(ii) For any $\varepsilon>0$ an operator%
\[
\left(  \Theta_{\varepsilon}z\right)  \left(  t\right)  =\mathfrak{T}\left(
t,0;F\left(  z\right)  \right)  x_{0}+\int_{0}^{t}\mathfrak{T}\left(
t,s;F\left(  z\right)  \right)  \left[  Bu_{\varepsilon}\left(  s,z\right)
+g\left(  s,z\left(  s\right)  \right)  \right]  ds
\]
admits a fixed point, $y_{\varepsilon}^{\ast}\left(  \cdot\right)  \in
C\left(  \left[  0,T\right]  ,\mathfrak{X}\right)  .$

Define
\begin{align*}
\mathfrak{T}^{\ast}\left(  t,s;F\right)  y  &  =\mathfrak{U}^{\ast}\left(
t-s\right)  y+\int_{s}^{t}\ \mathfrak{U}^{\ast}\left(  t-r\right)  F^{\ast
}\left(  r,z\left(  r\right)  \right)  \mathfrak{T}^{\ast}\left(
r,s;F\right)  ydr,\\
\Gamma_{0}^{T}\left(  F\right)  y  &  =\int_{0}^{T}\mathfrak{T}\left(
T,s;F\right)  BB^{\ast}\mathfrak{T}^{\ast}\left(  T,s;F\right)  yds.
\end{align*}

First we prove several lemmas.

\begin{lemma}
For any $G\in L^{2}\left(  \left[  0,T\right]  ,L\left(  \mathfrak{X}\right)
\right)  $ there exists a unique strongly continuous function $\mathfrak{T}%
:\Delta\rightarrow L\left(  \mathfrak{X}\right)  ,$ $\Delta=\left\{  \left(
t,s\right)  :0\leq s\leq t\leq T\right\}  $, such that
\begin{align*}
\mathfrak{T}\left(  t,t\right)   &  =I,\ \ \ t\in\left[  0,T\right]  ,\\
\mathfrak{T}\left(  t,r\right)  \mathfrak{T}\left(  r,s\right)   &
=\mathfrak{T}\left(  t,s\right)  ,\ \ 0\leq s\leq r\leq t\leq T,
\end{align*}%
\begin{align*}
\mathfrak{T}\left(  t,s;G\right)  y  &  =\mathfrak{U}\left(  t-s\right)
y+\int_{0}^{t}\mathfrak{U}\left(  t-r\right)  G\left(  r\right)
\mathfrak{T}\left(  r,s;G\right)  ydr\\
&  =\mathfrak{U}\left(  t-s\right)  y+\int_{0}^{t}\mathfrak{T}\left(
t,r;G\right)  G\left(  r\right)  \mathfrak{U}\left(  r-s\right)  ydr.
\end{align*}

\end{lemma}

The operator valued function $\mathfrak{T}:\Delta\rightarrow L\left(
\mathfrak{X}\right)  $ is the evolution operator generated by $A+F\left(
\cdot,z\left(  \cdot\right)  \right)  $.

Define%
\[
\mathfrak{T}^{\ast}\left(  T,t;G\right)  \eta=\mathfrak{U}^{\ast}\left(
T-t\right)  \eta+\int_{t}^{T}\mathfrak{U}^{\ast}\left(  T-r\right)  G^{\ast
}\left(  r\right)  \mathfrak{T}^{\ast}\left(  T,r;G\right)  \eta dr.
\]

\begin{lemma}
Suppose that $g\left(  t\right)  ,$ $\varphi\left(  t\right)  ,$ $\psi\left(
t\right)  \geq0$ and $\omega\left(  t\right)  \geq0$ are integrable functions.
If%
\[
g\left(  t\right)  \leq\varphi\left(  t\right)  +\psi\left(  t\right)
\int_{t}^{b}\omega\left(  r\right)  g\left(  r\right)  dr,
\]
then
\[
g\left(  t\right)  \leq\varphi\left(  t\right)  +\psi\left(  t\right)
\int_{t}^{b}\varphi\left(  r\right)  \omega\left(  r\right)  e^{\int_{r}%
^{t}\psi\left(  s\right)  \omega\left(  s\right)  ds}dr.
\]

\end{lemma}

\begin{lemma}
\label{lem3}Let $G_{n}\in L^{2}\left(  0,T;L\left(  \mathfrak{X}\right)
\right)  $ and $\eta_{n},\eta\in\mathfrak{X}$ such that%
\begin{equation}
\left\{
\begin{tabular}
[c]{lll}%
$G_{n}$ & is uniformly bounded in & $L^{2}\left(  0,T;L\left(  \mathfrak{X}%
\right)  \right)  ,$\\
$\eta_{n}\rightharpoonup\eta$ & weakly in & $\mathfrak{X},$ as $n\rightarrow
\infty,$%
\end{tabular}
\ \ \ \ \ \ \ \ \ \ \ \right.  \label{m2}%
\end{equation}
then there exists $G\in L^{2}\left(  0,T;L\left(  \mathfrak{X}\right)
\right)  $ such that
\[%
\begin{tabular}
[c]{lll}%
$\mathfrak{T}^{\ast}\left(  T,\cdot;G_{n}\right)  \eta_{n}\rightarrow
\mathfrak{T}^{\ast}\left(  T,\cdot;G\right)  \eta$ & in & $C\left(
0,T;\mathfrak{X}\right)  ,$\\
$\mathfrak{T}\left(  T,\cdot;G_{n}\right)  \eta_{n}\rightarrow\mathfrak{T}%
\left(  T,\cdot;G\right)  \eta$ & in & $C\left(  0,T;\mathfrak{X}\right)  ,$%
\end{tabular}
\ \ \ \ \ \ \ \ \ \ \
\]
as $n\rightarrow\infty.$
\end{lemma}

\begin{proof}
Let $\left\{  e_{m}:m\geq1\right\}  $ be a basis of $\mathfrak{X}.$ By our
assumption, there exists $C>0$ such that for all $n\geq1$
\[
\int_{0}^{T}\left\Vert G_{n}\left(  t\right)  \right\Vert _{L\left(
\mathfrak{X}\right)  }^{2}dt\leq C.
\]
It follows that%
\[
\int_{0}^{T}\left\Vert G_{n}\left(  t\right)  e_{m}\right\Vert _{\mathfrak{X}%
}^{2}dt\leq C.
\]
By the \textquotedblright diagonal argument\textquotedblright, we know that
there exists a subsequence, denoted again by $\left\{  G_{n}\left(
\cdot\right)  e_{m}:n\geq1\right\}  $, which is weakly convergent in
$L^{2}\left(  0,T;\mathfrak{X}\right)  $ for all $m\geq1$. Since $\left\{
e_{m}:m\geq1\right\}  $ is dense in $\mathfrak{X}$, we know that the sequence
$\left\{  G_{n}\left(  \cdot\right)  x\right\}  $ is weakly convergent in
$L^{2}\left(  0,T;\mathfrak{X}\right)  $ for all $x\in\mathfrak{X}$ to some
$G\left(  \cdot\right)  x\in L^{2}\left(  0,T;\mathfrak{X}\right)  .$ It is
clear that $G\left(  \cdot\right)  \in L^{2}\left(  0,T;L\left(
\mathfrak{X}\right)  \right)  .$

Denote%
\[
\xi_{n}\left(  t\right)  =\mathfrak{T}^{\ast}\left(  T,t;G_{n}\right)
\eta_{n},\ \ \xi\left(  t\right)  =\mathfrak{T}^{\ast}\left(  T,t;G\right)
\eta,\ \ t\in\left[  0,T\right]  .
\]
It is easily seen that%
\begin{equation}
\xi_{n}\left(  t\right)  =\mathfrak{U}^{\ast}\left(  T-t\right)  \eta_{n}%
+\int_{t}^{T}\mathfrak{U}^{\ast}\left(  T-r\right)  G_{n}\left(  r\right)
\xi_{n}\left(  r\right)  dr,\ \ t\in\left[  0,T\right]  . \label{m3}%
\end{equation}
Then by (\ref{m3}) and the Gronwall inequality,%
\[
\left\Vert \xi_{n}\left(  t\right)  \right\Vert \leq M\left\Vert \eta
_{n}\right\Vert +M\int_{t}^{T}\left\Vert G_{n}\left(  r\right)  \right\Vert
_{L\left(  \mathfrak{X}\right)  }\left\Vert \xi_{n}\left(  r\right)
\right\Vert dr,\ \ t\geq0,
\]
we have%
\begin{align}
\left\Vert \xi_{n}\left(  t\right)  \right\Vert  &  \leq M\left\Vert \eta
_{n}\right\Vert +M^{2}\left\Vert \eta_{n}\right\Vert \nonumber\\
&  \times\int_{t}^{T}\left\Vert G_{n}\left(  r\right)  \right\Vert _{L\left(
\mathfrak{X}\right)  }\exp\left(  M\int_{r}^{t}\left\Vert G_{n}\left(
s\right)  \right\Vert _{L\left(  \mathfrak{X}\right)  }ds\right)  dr.
\label{m33}%
\end{align}
From (\ref{m2}), we have the uniform boundedness of $\left\{  \eta
_{n}\right\}  .$ So from (\ref{m33}) one obtains the uniform boundedness of
$\left\{  \xi_{n}\left(  \cdot\right)  \right\}  $ in $C\left(
0,T;\mathfrak{X}\right)  $ and the uniform boundedness of $\left\{
G_{n}\left(  \cdot\right)  \xi_{n}\left(  \cdot\right)  \right\}  $ in
$L^{2}\left(  0,T;\mathfrak{X}\right)  $. Thus, having in mind compactness of
$\mathfrak{U}\left(  t\right)  ,$ $t>0$, one can show that $\left\{  \xi
_{n}\left(  \cdot\right)  \right\}  $ is relatively compact in $C\left(
0,T;\mathfrak{X}\right)  $. Let $\xi\left(  \cdot\right)  $ be any limit point
of $\left\{  \xi_{n}\left(  \cdot\right)  \right\}  $ in $C\left(
0,T;\mathfrak{X}\right)  .$ On the other hand, for any $r\in\left[
0,T\right]  $%
\[
\mathfrak{U}^{\ast}\left(  T-r\right)  G_{n}\left(  r\right)  \xi_{n}\left(
r\right)  \rightarrow\mathfrak{U}^{\ast}\left(  T-r\right)  G\left(  r\right)
\xi\left(  r\right)  \ \text{in }\mathfrak{X}.
\]
Passing to the limit in (\ref{m3}) along some proper subsequence, we see that
$\xi\left(  \cdot\right)  $ satisfies%
\begin{equation}
\xi\left(  t\right)  =\mathfrak{U}^{\ast}\left(  T-t\right)  \eta+\int_{t}%
^{T}\mathfrak{U}^{\ast}\left(  T-r\right)  G\left(  r\right)  \xi\left(
r\right)  dr,\ \ t\in\left[  0,T\right]  . \label{m4}%
\end{equation}
By uniqueness of the solutions to (\ref{m4}), we obtain that the whole
sequence $\left\{  \xi_{n}\left(  \cdot\right)  \right\}  $ converges to
$\xi\left(  \cdot\right)  $ in $C\left(  0,T;\mathfrak{X}\right)  $.
Similarly, we may prove that
\[
\mathfrak{T}\left(  T,\cdot;G_{n}\right)  \eta_{n}\rightarrow\mathfrak{T}%
\left(  T,\cdot;G\right)  \eta\text{ in }C\left(  0,T;\mathfrak{X}\right)
\text{ as }n\rightarrow\infty.
\]

\end{proof}

\begin{lemma}
\label{lem4}Let $G_{n}\in L^{2}\left(  0,T;L\left(  \mathfrak{X}\right)
\right)  $ and $\eta_{n},\eta\in\mathfrak{X}$ such that%
\[
\left\{
\begin{tabular}
[c]{lll}%
$G_{n}$ & is uniformly bounded in & $L^{2}\left(  0,T;L\left(  \mathfrak{X}%
\right)  \right)  ,$\\
$\eta_{n}\rightharpoonup\eta$ & weakly in & $\mathfrak{X},$ as $n\rightarrow
\infty,$%
\end{tabular}
\ \ \ \ \ \ \ \ \ \right.
\]
then there exists $G\in L^{2}\left(  0,T;L\left(  \mathfrak{X}\right)
\right)  $ such that%
\[
\Gamma_{0}^{T}\left(  G_{n}\right)  \eta_{n}\rightarrow\Gamma_{0}^{T}\left(
G\right)  \eta\ \ \ \text{in\ \ }\mathfrak{X},\ \text{as\ \ }n\rightarrow
\infty,
\]
where%
\begin{align*}
\Gamma_{0}^{T}\left(  G_{n}\right)  \eta_{n}  &  =\int_{0}^{T}\mathfrak{T}%
\left(  T,s;G_{n}\right)  BB^{\ast}\mathfrak{T}^{\ast}\left(  T,s;G_{n}%
\right)  \eta_{n}ds,\\
\Gamma_{0}^{T}\left(  G\right)  \eta &  =\int_{0}^{T}\mathfrak{T}\left(
T,s;G\right)  BB^{\ast}\mathfrak{T}^{\ast}\left(  T,s;G\right)  \eta ds.
\end{align*}

\end{lemma}

\begin{proof}
The desired convergence follows from Lemma \ref{lem3}, boundednes of
$\mathfrak{T}\left(  T,s;G_{n}\right)  $ and from the following inequality%
\begin{align*}
&  \left\Vert \Gamma_{0}^{T}\left(  G_{n}\right)  \eta_{n}-\Gamma_{0}%
^{T}\left(  G\right)  \eta\right\Vert \\
&  =\left\Vert \int_{0}^{T}\mathfrak{T}\left(  T,s;G_{n}\right)  BB^{\ast
}\mathfrak{T}^{\ast}\left(  T,s;G_{n}\right)  \eta_{n}ds-\int_{0}%
^{T}\mathfrak{T}\left(  T,s;G\right)  BB^{\ast}\mathfrak{T}^{\ast}\left(
T,s;G\right)  \eta ds\right\Vert \\
&  \leq\left\Vert \int_{0}^{T}\mathfrak{T}\left(  T,s;G_{n}\right)  BB^{\ast
}\mathfrak{T}^{\ast}\left(  T,s;G_{n}\right)  \eta_{n}ds-\int_{0}%
^{T}\mathfrak{T}\left(  T,s;G_{n}\right)  BB^{\ast}\mathfrak{T}^{\ast}\left(
T,s;G\right)  \eta ds\right\Vert \\
&  +\left\Vert \int_{0}^{T}\mathfrak{T}\left(  T,s;G_{n}\right)  BB^{\ast
}\mathfrak{T}^{\ast}\left(  T,s;G\right)  \eta ds-\int_{0}^{T}\mathfrak{T}%
\left(  T,s;G\right)  BB^{\ast}\mathfrak{T}^{\ast}\left(  T,s;G\right)  \eta
ds\right\Vert \\
&  \leq\int_{0}^{T}\left\Vert \mathfrak{T}\left(  T,s;G_{n}\right)  BB^{\ast
}\right\Vert \left\Vert \mathfrak{T}^{\ast}\left(  T,s;G_{n}\right)  \eta
_{n}-\mathfrak{T}^{\ast}\left(  T,s;G\right)  \eta\right\Vert ds\\
&  +\int_{0}^{T}\left\Vert \left[  \mathfrak{T}\left(  T,s;G_{n}\right)
-\mathfrak{T}\left(  T,s;G\right)  \right]  BB^{\ast}\mathfrak{T}^{\ast
}\left(  T,s;G\right)  \eta\right\Vert ds.
\end{align*}

\end{proof}

\begin{lemma}
\label{lem5}Let $G_{n}\in L^{2}\left(  0,T;L\left(  \mathfrak{X}\right)
\right)  $ such that%
\[%
\begin{tabular}
[c]{lll}%
$G_{n}$ & is uniformly bounded in & $L^{2}\left(  0,T;L\left(  \mathfrak{X}%
\right)  \right)  ,$%
\end{tabular}
\ \ \ \
\]
then there exists $G\in L^{2}\left(  0,T;L\left(  \mathfrak{X}\right)
\right)  $ such that
\begin{align}
\lim_{n\rightarrow\infty}\left\Vert \varepsilon\left(  \varepsilon
I+\Gamma_{0}^{T}\left(  G_{n}\right)  \right)  ^{-1}\pi_{M}-\varepsilon\left(
\varepsilon I+\Gamma_{0}^{T}\left(  G\right)  \right)  ^{-1}\pi_{M}%
\right\Vert  &  =0,\label{qq1}\\
\lim_{n\rightarrow\infty}\left\Vert \left(  \varepsilon\left(  I-\pi
_{M}\right)  +\Gamma_{0}^{T}\left(  G_{n}\right)  \right)  ^{-1}%
h\rightarrow\left(  \varepsilon\left(  I-\pi_{M}\right)  +\Gamma_{0}%
^{T}\left(  G\right)  \right)  ^{-1}h\right\Vert  &  =0, \label{qq2}%
\end{align}
$\ $for any $h\in\mathfrak{X.}$
\end{lemma}

\begin{proof}
Set $\gamma_{n}:=\left\Vert \varepsilon\left(  \varepsilon I+\Gamma_{0}%
^{T}\left(  G_{n}\right)  \right)  ^{-1}\pi_{M}-\varepsilon\left(  \varepsilon
I+\Gamma_{0}^{T}\left(  G\right)  \right)  ^{-1}\pi_{M}\right\Vert _{L\left(
\mathfrak{X}\right)  }.$ There exists $\left\{  h_{m}\in\mathfrak{X}%
:\left\Vert h_{m}\right\Vert =1\right\}  $ such that
\[
\gamma_{n,m}:=\left\Vert \left(  \varepsilon\left(  \varepsilon I+\Gamma
_{0}^{T}\left(  G_{n}\right)  \right)  ^{-1}\pi_{M}-\varepsilon\left(
\varepsilon I+\Gamma_{0}^{T}\left(  G\right)  \right)  ^{-1}\pi_{M}\right)
h_{m}\right\Vert ,\ \ \gamma_{n,m}\rightarrow\gamma_{n}%
\]
$\ $as $m\rightarrow\infty$. Since $\left\{  \pi_{M}h_{m}\right\}  $ is a
sequence of finite dimensional vectors and $\left\Vert \pi_{M}h_{m}\right\Vert
\leq1$, then there is a subsequence denoted by $\left\{  \pi_{M}h_{m}\right\}
$again, such that $\pi_{M}h_{m}\rightarrow h_{0}\in M\ $as $m\rightarrow
\infty$. It follows that
\begin{align*}
z_{n,m}  &  :=\left(  \varepsilon\left(  \varepsilon I+\Gamma_{0}^{T}\left(
G_{n}\right)  \right)  ^{-1}\pi_{M}-\varepsilon\left(  \varepsilon
I+\Gamma_{0}^{T}\left(  G\right)  \right)  ^{-1}\pi_{M}\right)  h_{m}\\
&  =\varepsilon\left(  \varepsilon I+\Gamma_{0}^{T}\left(  G_{n}\right)
\right)  ^{-1}\left(  \Gamma_{0}^{T}\left(  G\right)  -\Gamma_{0}^{T}\left(
G_{n}\right)  \right)  \left(  \varepsilon I+\Gamma_{0}^{T}\left(  G\right)
\right)  ^{-1}\pi_{M}h_{m}\\
&  \rightarrow\varepsilon\left(  \varepsilon I+\Gamma_{0}^{T}\left(
G_{n}\right)  \right)  ^{-1}\left(  \Gamma_{0}^{T}\left(  G\right)
-\Gamma_{0}^{T}\left(  G_{n}\right)  \right)  \left(  \varepsilon I+\Gamma
_{0}^{T}\left(  G\right)  \right)  h_{0}:=z_{n}\ \
\end{align*}
as\ $m\rightarrow\infty,$ and $\lim_{m\rightarrow\infty}\gamma_{n,m}%
=\left\Vert z_{n}\right\Vert =\gamma_{n},\ \ \ \left\Vert z_{n,m}\right\Vert
\leq2.$ By Lemma \ref{lem4} we have%
\begin{align*}
\gamma_{n}  &  =\left\Vert z_{n}\right\Vert =\left\Vert \varepsilon\left(
\varepsilon I+\Gamma_{0}^{T}\left(  G_{n}\right)  \right)  ^{-1}\left(
\Gamma_{0}^{T}\left(  G\right)  -\Gamma_{0}^{T}\left(  G_{n}\right)  \right)
\left(  \varepsilon I+\Gamma_{0}^{T}\left(  G\right)  \right)  ^{-1}%
h_{0}\right\Vert \\
&  \leq\left\Vert \varepsilon\left(  \varepsilon I+\Gamma_{0}^{T}\left(
G_{n}\right)  \right)  ^{-1}\right\Vert \left\Vert \left(  \Gamma_{0}%
^{T}\left(  G\right)  -\Gamma_{0}^{T}\left(  G_{n}\right)  \right)  \left(
\varepsilon I+\Gamma_{0}^{T}\left(  G\right)  \right)  ^{-1}h_{0}\right\Vert
\\
&  \leq\left\Vert \left(  \Gamma_{0}^{T}\left(  G\right)  -\Gamma_{0}%
^{T}\left(  G_{n}\right)  \right)  \left(  \varepsilon I+\Gamma_{0}^{T}\left(
G\right)  \right)  ^{-1}h_{0}\right\Vert \rightarrow0\ \ \text{as
}n\rightarrow\infty.
\end{align*}
So $\lim_{n\rightarrow\infty}\lim_{m\rightarrow\infty}\gamma_{n,m}%
=\lim_{n\rightarrow\infty}\gamma_{n}=0$.

For every $h\in\mathfrak{X}$ we have%
\begin{align*}
&  \left(  \varepsilon\left(  I-\pi_{M}\right)  +\Gamma_{0}^{T}\left(
G\right)  \right)  ^{-1}h-\left(  \varepsilon\left(  I-\pi_{M}\right)
+\Gamma_{0}^{T}\left(  G_{n}\right)  \right)  ^{-1}h\\
&  =\left(  \varepsilon\left(  I-\pi_{M}\right)  +\Gamma_{0}^{T}\left(
G_{n}\right)  \right)  ^{-1}\\
&  \times\left[  \left(  \varepsilon\left(  I-\pi_{M}\right)  +\Gamma_{0}%
^{T}\left(  G\right)  -\Gamma_{0}^{T}\left(  G\right)  +\Gamma_{0}^{T}\left(
G_{n}\right)  \right)  \left(  \varepsilon\left(  I-\pi_{M}\right)
+\Gamma_{0}^{T}\left(  G\right)  \right)  ^{-1}-I\right]  h\\
&  =\left(  \varepsilon\left(  I-\pi_{M}\right)  +\Gamma_{0}^{T}\left(
G_{n}\right)  \right)  ^{-1}\left[  I+\left(  \Gamma_{0}^{T}\left(
G_{n}\right)  -\Gamma_{0}^{T}\left(  G\right)  \right)  \left(  \varepsilon
\left(  I-\pi_{M}\right)  +\Gamma_{0}^{T}\left(  G\right)  \right)
^{-1}-I\right]  h\\
&  =\left(  \varepsilon\left(  I-\pi_{M}\right)  +\Gamma_{0}^{T}\left(
G_{n}\right)  \right)  ^{-1}\left(  \Gamma_{0}^{T}\left(  G_{n}\right)
-\Gamma_{0}^{T}\left(  G\right)  \right)  \left(  \varepsilon\left(  I-\pi
_{M}\right)  +\Gamma_{0}^{T}\left(  G\right)  \right)  ^{-1}h.
\end{align*}
By (\ref{qq1}) and Theorem \ref{thm:2} we have%
\begin{align}
&  \lim_{n\rightarrow\infty}\left\Vert \left(  \varepsilon\left(  I-\pi
_{M}\right)  +\Gamma_{0}^{T}\left(  G_{n}\right)  \right)  ^{-1}\right\Vert
\nonumber\\
&  \leq\lim_{n\rightarrow\infty}\left\Vert \left(  I-\varepsilon\left(
\varepsilon I+\Gamma_{0}^{T}\left(  G_{n}\right)  \right)  ^{-1}\pi
_{M}\right)  ^{-1}\right\Vert \left\Vert \left(  \varepsilon I+\Gamma_{0}%
^{T}\left(  G_{n}\right)  \right)  ^{-1}\right\Vert \nonumber\\
&  \leq\frac{1}{\varepsilon}\frac{1}{1-\lim_{n\rightarrow\infty}\left\Vert
\varepsilon\left(  \varepsilon I+\Gamma_{0}^{T}\left(  G_{n}\right)  \right)
^{-1}\pi_{M}\right\Vert }\label{qw1}\\
&  =\frac{1}{\varepsilon}\frac{1}{1-\left\Vert \varepsilon\left(  \varepsilon
I+\Gamma_{0}^{T}\left(  G\right)  \right)  ^{-1}\pi_{M}\right\Vert }%
:=\delta\left(  \varepsilon\right) \nonumber
\end{align}
Now desired convergence (\ref{qw1}) follows from Lemma \ref{lem4}:
\begin{align*}
&  \lim_{n\rightarrow\infty}\left\Vert \left(  \varepsilon\left(  I-\pi
_{M}\right)  +\Gamma_{0}^{T}\left(  G\right)  \right)  ^{-1}h-\left(
\varepsilon\left(  I-\pi_{M}\right)  +\Gamma_{0}^{T}\left(  G_{n}\right)
\right)  ^{-1}h\right\Vert \\
&  \leq\lim_{n\rightarrow\infty}\left\Vert \left(  \varepsilon\left(
I-\pi_{M}\right)  +\Gamma_{0}^{T}\left(  G_{n}\right)  \right)  ^{-1}%
\right\Vert \\
&  \times\lim_{n\rightarrow\infty}\left\Vert \left(  \Gamma_{0}^{T}\left(
G_{n}\right)  -\Gamma_{0}^{T}\left(  G\right)  \right)  \left(  \varepsilon
\left(  I-\pi_{M}\right)  +\Gamma_{0}^{T}\left(  G\right)  \right)
^{-1}h\right\Vert \\
&  \leq\delta\left(  \varepsilon\right)  \lim_{n\rightarrow\infty}\left\Vert
\left(  \Gamma_{0}^{T}\left(  G_{n}\right)  -\Gamma_{0}^{T}\left(  G\right)
\right)  \left(  \varepsilon\left(  I-\pi_{M}\right)  +\Gamma_{0}^{T}\left(
G\right)  \right)  ^{-1}h\right\Vert =0.
\end{align*}

\end{proof}

\begin{lemma}
\label{lem2}Let $z\left(  \cdot\right)  \in C\left(  0,T;\mathfrak{X}\right)
$ and $\mathfrak{T}\left(  t,s;F\left(  z\right)  \right)  $ be the evolution
operator generated by $A+F\left(  z\right)  ,$ where $F$ is defined by
(\ref{p2}) and let%
\begin{align*}
u_{\varepsilon}\left(  t,z\right)   &  =B^{\ast}\mathfrak{T}^{\ast}\left(
T,t;F\left(  z\right)  \right)  \left(  \varepsilon\left(  I-\pi_{M}\right)
+\Gamma_{0}^{T}\left(  F\left(  z\right)  \right)  \right)  ^{-1}\\
&  \times\left(  h-\mathfrak{T}\left(  T,0;F\left(  z\right)  \right)
x_{0}-\int_{0}^{T}\mathfrak{T}\left(  T,s;F\left(  z\right)  \right)  g\left(
s,z\left(  s\right)  \right)  ds\right)  .
\end{align*}
The control $z\rightarrow u_{\varepsilon}\left(  t,z\right)  :C\left(
0,T;\mathfrak{X}\right)  \rightarrow C\left(  0,T;\mathfrak{X}\right)  $ is
continuous and
\begin{gather*}
\left\Vert u_{\varepsilon}\left(  t,z\right)  \right\Vert \leq R_{\varepsilon
}:=\dfrac{1}{\varepsilon\left(  1-\gamma_{\varepsilon}\right)  }%
M_{B}M_{\mathfrak{T}}\left(  \left\Vert h\right\Vert +M_{\mathfrak{T}%
}\left\Vert x_{0}\right\Vert +M_{\mathfrak{T}}T\left\Vert g\right\Vert
_{C}\right)  ,\\
\gamma_{\varepsilon}:=\sup_{z\in C\left(  0,T;\mathfrak{X}\right)  }\left\Vert
\varepsilon\left(  \varepsilon I+\Gamma_{0}^{T}\left(  F\left(  z\right)
\right)  \right)  ^{-1}\pi_{M}\right\Vert <1,\\
M_{\mathfrak{T}}:=\sup\left\{  \left\Vert \mathfrak{T}\left(  t,s;F\left(
z\right)  \right)  \right\Vert :0\leq s\leq t\leq T\right\}  .
\end{gather*}

\end{lemma}

\begin{proof}
To prove continuity of $u_{\varepsilon}\left(  \cdot,z\right)  ,$ let
$\left\{  z_{n}\right\}  \subset C\left(  0,T;\mathfrak{X}\right)  $ with
$z_{n}\rightarrow z$ in $C\left(  0,T;\mathfrak{X}\right)  $. By assumption
(A2) and (A3) the functions $F\left(  z\right)  $ and $g\left(  s,z\left(
s\right)  \right)  $ are continuous. It follows that $\mathfrak{T}\left(
T,s;F\left(  z\right)  \right)  $ and%
\[
h\left(  z\right)  :=h-\mathfrak{T}\left(  T,0;F\left(  z\right)  \right)
x_{0}-\int_{0}^{T}\mathfrak{T}\left(  T,s;F\left(  z\right)  \right)  g\left(
s,z\left(  s\right)  \right)  ds
\]
are continuous in $z$. Then from the following equality%
\begin{align*}
&  u_{\varepsilon}\left(  t,z_{n}\right)  -u_{\varepsilon}\left(  t,z\right)
\\
&  =B^{\ast}\mathfrak{T}^{\ast}\left(  T,t;F\left(  z_{n}\right)  \right)
\left(  \varepsilon\left(  I-\pi_{M}\right)  +\Gamma_{0}^{T}\left(  F\left(
z_{n}\right)  \right)  \right)  ^{-1}\left(  h\left(  z_{n}\right)  -h\left(
z\right)  \right) \\
&  +B^{\ast}\mathfrak{T}^{\ast}\left(  T,t;F\left(  z_{n}\right)  \right) \\
&  \times\left[  \left(  \varepsilon\left(  I-\pi_{M}\right)  +\Gamma_{0}%
^{T}\left(  F\left(  z_{n}\right)  \right)  \right)  ^{-1}-\left(
\varepsilon\left(  I-\pi_{M}\right)  +\Gamma_{0}^{T}\left(  F\left(  z\right)
\right)  \right)  ^{-1}\right]  h\left(  z\right) \\
&  +\left[  B^{\ast}\mathfrak{T}^{\ast}\left(  T,t;F\left(  z_{n}\right)
\right)  -B^{\ast}\mathfrak{T}^{\ast}\left(  T,t;F\left(  z\right)  \right)
\right]  \left(  \varepsilon\left(  I-\pi_{M}\right)  +\Gamma_{0}^{T}\left(
F\left(  z\right)  \right)  \right)  ^{-1}h\left(  z\right)
\end{align*}
it follows that $u_{\varepsilon}\left(  t,z_{n}\right)  \rightarrow
u_{\varepsilon}\left(  t,z\right)  $ as $n\rightarrow\infty$ in $C\left(
0,T;\mathfrak{X}\right)  $. Moreover,%
\begin{align*}
\left\Vert u_{\varepsilon}\left(  t,z\right)  \right\Vert  &  \leq
M_{B}M_{\mathfrak{T}}\dfrac{1}{1-\left\Vert \varepsilon\left(  \varepsilon
I+\Gamma_{0}^{T}\left(  F\left(  z\right)  \right)  \right)  ^{-1}\pi
_{M}\right\Vert }\\
&  \times\frac{1}{\varepsilon}\left\Vert h-\mathfrak{T}\left(  T,0;F\left(
z\right)  \right)  x_{0}-\int_{0}^{T}\mathfrak{T}\left(  T,s;F\left(
z\right)  \right)  g\left(  s,z\left(  s\right)  \right)  ds\right\Vert \\
&  \leq\dfrac{1}{\varepsilon\left(  1-\gamma_{\varepsilon}\right)  }%
M_{B}M_{\mathfrak{T}}\left(  \left\Vert h\right\Vert +M_{\mathfrak{T}%
}\left\Vert x_{0}\right\Vert +M_{\mathfrak{T}}T\left\Vert m\right\Vert
_{C}\right)  :=R_{\varepsilon},
\end{align*}
where
\[
\gamma_{\varepsilon}=\sup_{z\in C\left(  0,T;\mathfrak{X}\right)  }\left\Vert
\varepsilon\left(  \varepsilon I+\Gamma_{0}^{T}\left(  F\left(  z\right)
\right)  \right)  ^{-1}\pi_{M}\right\Vert .
\]
Show that $\gamma_{\varepsilon}<1.$ Contrary, assume that there exists a
sequence $\left\{  z_{n}\right\}  $ such that
\[
\lim_{n\rightarrow\infty}\gamma_{\varepsilon}\left(  z_{n}\right)
=1,\ \ \ \ \ \gamma_{\varepsilon}\left(  z_{n}\right)  =\left\Vert
\varepsilon\left(  \varepsilon I+\Gamma_{0}^{T}\left(  F\left(  z_{n}\right)
\right)  \right)  ^{-1}\pi_{M}\right\Vert .
\]
Then since $\left\Vert F\left(  t,z_{n}\left(  t\right)  \right)  \right\Vert
\leq L,$ $F\left(  \cdot,z_{n}\left(  \cdot\right)  \right)  \in L^{2}\left(
0,T;L\left(  \mathfrak{X}\right)  \right)  ,$ then by Lemma \ref{lem5} there
exists $\widetilde{F}\in L^{2}\left(  0,T;L\left(  \mathfrak{X}\right)
\right)  $ such that%
\[
\lim_{n\rightarrow\infty}\gamma_{\varepsilon}\left(  z_{n}\right)
=\lim_{n\rightarrow\infty}\left\Vert \varepsilon\left(  \varepsilon
I+\Gamma_{0}^{T}\left(  F\left(  z_{n}\right)  \right)  \right)  ^{-1}\pi
_{M}\right\Vert =\left\Vert \varepsilon\left(  \varepsilon I+\Gamma_{0}%
^{T}\left(  \widetilde{F}\right)  \right)  ^{-1}\pi_{M}\right\Vert <1,
\]
which is contradiction.
\end{proof}

\begin{theorem}
\label{thm:5}For any $\varepsilon>0$ the operator $\Theta_{\varepsilon}\left(
z\right)  $ has a fixed point in $C\left(  0,T;\mathfrak{X}\right)  $.
\end{theorem}

\begin{proof}
Claim1. The operator $\Theta_{\varepsilon}\left(  z\right)  $ sends $C\left(
0,T;\mathfrak{X}\right)  $ into a bounded set.

We need to show that, for any $\varepsilon>0$ there exists $k\left(
\varepsilon\right)  >0$ such that $\left\Vert \left(  \Theta_{\varepsilon
}z\right)  \left(  t\right)  \right\Vert \leq k\left(  \varepsilon\right)  $
for all $z\left(  \cdot\right)  \in C\left(  0,T;\mathfrak{X}\right)  .$
Indeed,%
\begin{align*}
\left\Vert \left(  \Theta_{\varepsilon}z\right)  \left(  t\right)
\right\Vert  &  \leq\left\Vert \mathfrak{T}\left(  t,0;F\left(  z\right)
\right)  \right\Vert \left\Vert y_{0}\right\Vert \\
&  +\int_{0}^{t}\left\Vert \mathfrak{T}\left(  t,s;F\left(  z\right)  \right)
\right\Vert \left[  \left\Vert B\right\Vert \left\Vert u_{\varepsilon}\left(
s,z\right)  \right\Vert +\left\Vert g\left(  s,z\left(  s\right)  \right)
\right\Vert \right]  ds\\
&  \leq L\left\Vert y_{0}\right\Vert +LM_{B}\left(  R_{\varepsilon
}+T\left\Vert m\right\Vert _{C}\right)  =:k\left(  \varepsilon\right)  .
\end{align*}

Claim 2. The operator $\Theta_{\varepsilon}:C\left(  0,T;\mathfrak{X}\right)
\rightarrow C\left(  0,T;\mathfrak{X}\right)  $ is continuous.

Assume that the sequence $\left\{  z_{n}\right\}  \subset C\left(
0,T;\mathfrak{X}\right)  $ such that $z_{n}\rightarrow z$ in $C\left(
0,T;\mathfrak{X}\right)  .$ Then the triangle inequality we have%
\begin{align*}
\left\Vert \left(  \Theta_{\varepsilon}z_{n}\right)  \left(  t\right)
-\left(  \Theta_{\varepsilon}z\right)  \left(  t\right)  \right\Vert  &
\leq\left\Vert \left(  \mathfrak{T}\left(  t,0;F\left(  z_{n}\right)  \right)
y_{0}-\mathfrak{T}\left(  t,0;F\left(  z\right)  \right)  \right)
y_{0}\right\Vert \\
&  +\int_{0}^{t}\left\Vert \mathfrak{T}\left(  t,s;F\left(  z_{n}\right)
\right)  \right\Vert \left\Vert B\right\Vert \left\Vert u_{\varepsilon}\left(
s,z_{n}\right)  -u_{\varepsilon}\left(  s,z\right)  \right\Vert ds\\
&  +\int_{0}^{t}\left\Vert \mathfrak{T}\left(  t,0;F\left(  z_{n}\right)
\right)  -\mathfrak{T}\left(  t,0;F\left(  z\right)  \right)  \right\Vert
\left\Vert B\right\Vert \left\Vert u_{\varepsilon}\left(  s,z\right)
\right\Vert ds\\
&  +\int_{0}^{t}\left\Vert \mathfrak{T}\left(  t,s;F\left(  z\right)  \right)
\right\Vert \left\Vert B\right\Vert \left\Vert g\left(  s,z_{n}\left(
s\right)  \right)  -g\left(  s,z\left(  s\right)  \right)  \right\Vert ds\\
&  +\int_{0}^{t}\left\Vert \mathfrak{T}\left(  t,0;F\left(  z_{n}\right)
\right)  -\mathfrak{T}\left(  t,0;F\left(  z\right)  \right)  \right\Vert
\left\Vert g\left(  s,z\left(  s\right)  \right)  \right\Vert ds.
\end{align*}
Now from the continuity of $u_{\varepsilon}\left(  s,\cdot\right)  ,g\left(
\cdot\right)  ,\mathfrak{T}\left(  t,s;F\left(  \cdot\right)  \right)  $ we
get the desired continuity of $\Theta_{\varepsilon}$.

Claim 3. The family of functions $\left\{  \Theta_{\varepsilon}z:z\in C\left(
0,T;\mathfrak{X}\right)  \right\}  $ is equicontinuous.

To see this we fix $t_{1}>0$ and let $t_{2}>t_{1}$ and $\eta>0$ be small
enough. Then%
\begin{gather*}
\left\Vert \left(  \Theta_{\varepsilon}z\right)  \left(  t_{2}\right)
-\left(  \Theta_{\varepsilon}z\right)  \left(  t_{1}\right)  \right\Vert
\leq\left\Vert \left(  \mathfrak{T}\left(  t_{2},0;F\left(  z_{n}\right)
\right)  -\mathfrak{T}\left(  t_{1},0;F\left(  z\right)  \right)  \right)
y_{0}\right\Vert \\
+\int_{0}^{t_{1}-\eta}\left\Vert \mathfrak{T}\left(  t_{2},s;F\left(
z\right)  \right)  -\mathfrak{T}\left(  t_{1},s;F\left(  z\right)  \right)
\right\Vert \left\Vert Bu_{\varepsilon}\left(  s,z\right)  +g\left(
s,z\left(  s\right)  \right)  \right\Vert ds\\
+\int_{t_{1}-\eta}^{t_{1}}\left\Vert \mathfrak{T}\left(  t_{2},s;F\left(
z\right)  \right)  -\mathfrak{T}\left(  t_{1},s;F\left(  z\right)  \right)
\right\Vert \left\Vert Bu_{\varepsilon}\left(  s,z\right)  +g\left(
s,z\left(  s\right)  \right)  \right\Vert ds\\
+\int_{t_{1}}^{t_{2}}\left\Vert \mathfrak{T}\left(  t_{2},s;F\left(  z\right)
\right)  \right\Vert \left\Vert Bu_{\varepsilon}\left(  s,z\right)  +g\left(
s,z\left(  s\right)  \right)  \right\Vert ds.
\end{gather*}
We know that compactness of the evolution family $\mathfrak{T}\left(
t,s\right)  ,t-s>0$ implies the uniform continuity of $\mathfrak{T}\left(
t,s\right)  ,t-s>0.$ Having in mind that
\[
\left\Vert Bu_{\varepsilon}\left(  s,z\right)  +g\left(  s,z\left(  s\right)
\right)  \right\Vert \leq M_{B}R_{\varepsilon}+\left\Vert m\right\Vert
_{C},\ \ \
\]
we see that $\left\Vert \left(  \Theta_{\varepsilon}z\right)  \left(
t_{2}\right)  -\left(  \Theta_{\varepsilon}z\right)  \left(  t_{1}\right)
\right\Vert $ tends to zero independently of $z\in C\left(  0,T;\mathfrak{X}%
\right)  $ as $t_{2}-t_{1}\rightarrow0.$ It can be easily shown that the
family $\left\{  \Theta_{\varepsilon}z:z\in C\left(  0,T;\mathfrak{X}\right)
\right\}  $ is equicontinuous at $t=0$. Hence $\left\{  \Theta_{\varepsilon
}z:z\in C\left(  0,T;\mathfrak{X}\right)  \right\}  $ is equicontinuous.

Claim 4. The set $V\left(  t\right)  =\left\{  \left(  \Theta_{\varepsilon
}z\right)  \left(  t\right)  :z\left(  \cdot\right)  \in C\left(
0,T;\mathfrak{X}\right)  \right\}  $ is relatively compact in $\mathfrak{X}$.

Obviously, $V\left(  0\right)  $ is relatively compact in $\mathfrak{X}$. Let
$0<t\leq T$ be fixed and $0<\delta<t$. For $z\left(  \cdot\right)  \in
C\left(  0,T;\mathfrak{X}\right)  $ we define%
\[
\left(  \Theta_{\varepsilon}z\right)  \left(  t\right)  =\mathfrak{T}\left(
t,0;F\left(  z\right)  \right)  x_{0}+\int_{0}^{t}\mathfrak{T}\left(
t,s;F\left(  z\right)  \right)  \left[  Bu_{\varepsilon}\left(  s,z\right)
+g\left(  s,z\left(  s\right)  \right)  \right]  ds
\]%
\begin{align*}
\left(  \Theta_{\varepsilon}^{\delta}z\right)  \left(  t\right)   &
=\mathfrak{T}\left(  t,0;F\left(  z\right)  \right)  x_{0}\\
&  +\mathfrak{T}\left(  t,t-\delta;F\left(  z\right)  \right)  \int
_{0}^{t-\delta}\mathfrak{T}\left(  t-\delta,s;F\left(  z\right)  \right)
\left[  Bu_{\varepsilon}\left(  s,z\right)  +g\left(  s,z\left(  s\right)
\right)  \right]  ds.
\end{align*}
Then from the compactness of $\mathfrak{T}\left(  t,t-\delta;F\left(
z\right)  \right)  ,\delta>0$ we obtain that%
\[
V^{\delta}\left(  t\right)  =\left\{  \left(  \Theta_{\varepsilon}^{\delta
}z\right)  \left(  t\right)  :z\left(  \cdot\right)  \in C\left(
0,T;\mathfrak{X}\right)  \right\}
\]
is relatively compact in $\mathfrak{X}$ for every $\delta,0<\delta<t$.
Moreover, for every $z\left(  \cdot\right)  \in C\left(  0,T;\mathfrak{X}%
\right)  $ we have%
\begin{align*}
&  \left\Vert \left(  \Theta_{\varepsilon}z\right)  \left(  t\right)  -\left(
\Theta_{\varepsilon}^{\delta}z\right)  \left(  t\right)  \right\Vert \\
&  \leq\int_{t-\delta}^{t}\left\Vert \mathfrak{T}\left(  t,s;F\left(
z\right)  \right)  \left[  Bu_{\varepsilon}\left(  s,z\right)  +g\left(
s,z\left(  s\right)  \right)  \right]  \right\Vert ds.
\end{align*}
Therefore, there are relatively compact sets arbitrarily close to the set
$V\left(  t\right)  $. Hence $V\left(  t\right)  $ is also relatively compact
in $\mathfrak{X}$.

Thus thanks to the Arzela-Ascoli theorem, the operator $\Theta_{\varepsilon}$
is a compact operator for any $\varepsilon>0$. Consequently, the operator
$\Theta_{\varepsilon}:C\left(  0,T;\mathfrak{X}\right)  \rightarrow C\left(
0,T;\mathfrak{X}\right)  $ is continuous and compact with uniformly bounded
image. By the Schauder fixed point theorem, the operator $\Theta_{\varepsilon
}$ has at least one fixed point in $C\left(  0,T;\mathfrak{X}\right)  $.
\end{proof}

\subsection{Finite-approximate controllability}

Assume that $y_{\varepsilon}^{\ast}\left(  \cdot\right)  $ is a fixed point of
$\Theta_{\varepsilon}.$ We will show that the fixed point $y_{\varepsilon
}^{\ast}\left(  \cdot\right)  $ and the corresponding control $u_{\varepsilon
}^{\ast}\left(  \cdot\right)  =u^{\ast}\left(  \cdot,y_{\varepsilon}^{\ast
}\right)  $ satisfies%
\[
\left\Vert y_{\varepsilon}^{\ast}\left(  T;y_{0},u_{\varepsilon}^{\ast
}\right)  -y_{f}\right\Vert <\varepsilon,\ \ \ \pi_{M}y_{\varepsilon}^{\ast
}\left(  T;y_{0},u_{\varepsilon}^{\ast}\right)  =\pi_{M}y_{f}\ \text{for any
}y_{f}\in\mathfrak{X.}%
\]
Define%

\begin{align*}
h\left(  y_{\varepsilon}^{\ast}\right)   &  =y_{f}-\mathfrak{T}\left(
T,0;F\left(  y_{\varepsilon}^{\ast}\right)  \right)  y_{0}-\int_{0}%
^{T}\mathfrak{T}\left(  T,s;F\left(  y_{\varepsilon}^{\ast}\right)  \right)
g\left(  s,y_{\varepsilon}^{\ast}\left(  s\right)  \right)  ds,\\
\varphi_{\varepsilon}  &  =\left(  \varepsilon\left(  I-\pi_{M}\right)
+\Gamma_{0}^{T}\left(  F\left(  y_{\varepsilon}^{\ast}\right)  \right)
\right)  ^{-1}\left(  h\left(  y_{\varepsilon}^{\ast}\right)  \right)  .
\end{align*}
Now we are ready to state and prove the main result on finite-approximate
controllability of semilinear evolution systems in this paper.

\begin{theorem}
\label{thm1}Let (S), (F), (G), and (AC) hold. Then the system (\ref{mp1}) is
finite-approximately controllable on $\left[  0,T\right]  $.
\end{theorem}

\begin{proof}
Let $y_{\varepsilon}^{\ast}\in C\left(  0,T;\mathfrak{X}\right)  $ be a fixed
point of $\Theta_{\varepsilon}.$ Then we have
\[
y_{\varepsilon}^{\ast}\left(  T\right)  -y_{f}=-\varepsilon\left(  I-\pi
_{M}\right)  \varphi_{\varepsilon}=-\varepsilon\left(  I-\pi_{M}\right)
\left(  \varepsilon\left(  I-\pi_{M}\right)  +\Gamma_{0}^{T}\left(  F\left(
y_{\varepsilon}^{\ast}\right)  \right)  \right)  ^{-1}\left(  h\left(
y_{\varepsilon}^{\ast}\right)  \right)  .
\]
By the assumption (A2), $\left\Vert F\left(  s,y_{\varepsilon}^{\ast}\left(
s\right)  \right)  \right\Vert _{L\left(  \mathfrak{X}\right)  }\leq L.$ Then
by Lemma \ref{lem3} there exists $\widetilde{F}\in L^{2}\left(  0,T;L\left(
\mathfrak{X}\right)  \right)  $ such that
\[
\mathfrak{T}\left(  T,0;F\left(  y_{\varepsilon}^{\ast}\right)  \right)
y_{0}\rightarrow\mathfrak{T}\left(  T,0;\widetilde{F}\right)  y_{0}.
\]
On the other hand by the assumption (A3), $\left\Vert g\left(
s,y_{\varepsilon}^{\ast}\left(  s\right)  \right)  \right\Vert \leq m\left(
s\right)  $, and Dunford-Pettis Theorem, we have that the sequence $\left\{
g\left(  s,y_{\varepsilon}^{\ast}\left(  s\right)  \right)  \right\}  $ is
weakly compact in $L^{2}\left(  0,T;\mathfrak{X}\right)  $, so there is a
subsequence, still denoted by $\left\{  g\left(  s,y_{\varepsilon}^{\ast
}\left(  s\right)  \right)  \right\}  $ that weakly converges to, say, $g$ in
$L^{2}\left(  0,T;\mathfrak{X}\right)  $. Denote
\begin{align*}
\widetilde{h}  &  =y_{f}-\mathfrak{T}\left(  T,0;\widetilde{F}\right)
y_{0}-\int_{0}^{T}\mathfrak{T}\left(  T,s;\widetilde{F}\right)  g\left(
s\right)  ds,\\
\varphi_{\varepsilon}\left(  t,s\right)   &  =\mathfrak{T}\left(  t,s;F\left(
y_{\varepsilon}^{\ast}\right)  \right)  g\left(  s,y_{\varepsilon}^{\ast
}\left(  s\right)  \right)  -\mathfrak{T}\left(  t,s;\widetilde{F}\right)
g\left(  s\right)  ,\\
\psi_{\varepsilon}\left(  t,s\right)   &  =\int_{s}^{t}\mathfrak{U}\left(
t-r\right)  \left[  F\left(  r,y_{\varepsilon}^{\ast}\left(  r\right)
\right)  -\widetilde{F}\left(  r\right)  \right]  \mathfrak{T}\left(
r,s;\widetilde{F}\left(  r\right)  \right)  g\left(  s\right)  dr.
\end{align*}
Then by the definition of $\mathfrak{T}$ and the Gronwall inequality we have%
\begin{align*}
\varphi_{\varepsilon}\left(  t,s\right)   &  =\psi_{\varepsilon}\left(
t,s\right)  +\int_{s}^{t}\mathfrak{U}\left(  t-r\right)  F\left(
r,y_{\varepsilon}^{\ast}\left(  r\right)  \right)  \varphi\left(  r\right)
dr\\
\left\Vert \varphi_{\varepsilon}\left(  t,s\right)  \right\Vert  &
\leq\left\Vert \psi_{\varepsilon}\left(  t,s\right)  \right\Vert +M\int
_{s}^{t}\left\Vert \varphi_{\varepsilon}\left(  r,s\right)  \right\Vert
dr\Longrightarrow\\
\left\Vert \varphi_{\varepsilon}\left(  t,s\right)  \right\Vert  &
\leq\left\Vert \psi_{\varepsilon}\left(  t,s\right)  \right\Vert +M\int
_{s}^{t}\left\Vert \psi_{\varepsilon}\left(  r,s\right)  \right\Vert dr.
\end{align*}
It is clear that
\[
\psi_{\varepsilon}\left(  r,s\right)  \text{ is uniformly bounded and
}\left\Vert \psi_{\varepsilon}\left(  r,s\right)  \right\Vert \rightarrow0.\
\]
Then%
\[
\varphi_{\varepsilon}\left(  T,s\right)  =\mathfrak{T}\left(  T,s;F\left(
y_{\varepsilon}^{\ast}\right)  \right)  g\left(  s,y_{\varepsilon}^{\ast
}\left(  s\right)  \right)  -\mathfrak{T}\left(  T,s;\widetilde{F}\right)
g\left(  s\right)  \rightarrow0\ \text{as }\varepsilon\rightarrow0^{+}.
\]
So%
\begin{align}
\left\Vert h\left(  y_{\varepsilon}^{\ast}\right)  -\widetilde{h}\right\Vert
&  \leq\left\Vert \mathfrak{T}\left(  T,0;F\left(  y_{\varepsilon}^{\ast
}\right)  \right)  y_{0}-\mathfrak{T}\left(  T,0;\widetilde{F}\right)
y_{0}\right\Vert \nonumber\\
&  +\int_{0}^{T}\left\Vert \mathfrak{T}\left(  T,s;F\left(  y_{\varepsilon
}^{\ast}\right)  \right)  g\left(  s,y_{\varepsilon}^{\ast}\left(  s\right)
\right)  -\mathfrak{T}\left(  T,s;\widetilde{F}\right)  g\left(  s\right)
\right\Vert ds\nonumber\\
&  \rightarrow0\ \ \ \text{as\ \ }\varepsilon\rightarrow0^{+}. \label{ww3}%
\end{align}
To prove the strong convergence recall that
\begin{gather}
\left\Vert y_{\varepsilon}^{\ast}\left(  T\right)  -y_{f}\right\Vert
\leq\varepsilon\left\Vert \left(  I-\pi_{M}\right)  \left(  \varepsilon\left(
I-\pi_{M}\right)  +\Gamma_{0}^{T}\left(  F\left(  y_{\varepsilon}^{\ast
}\right)  \right)  \right)  ^{-1}\left(  h\left(  y_{\varepsilon}^{\ast
}\right)  -\widetilde{h}\right)  \right\Vert \nonumber\\
+\varepsilon\left\Vert \left(  I-\pi_{M}\right)  \left[  \left(
\varepsilon\left(  I-\pi_{M}\right)  +\Gamma_{0}^{T}\left(  F\left(
y_{\varepsilon}^{\ast}\right)  \right)  \right)  ^{-1}-\left(  \varepsilon
\left(  I-\pi_{M}\right)  +\Gamma_{0}^{T}\left(  \widetilde{F}\right)
\right)  ^{-1}\right]  \left(  \widetilde{h}\right)  \right\Vert
\label{ww11}\\
+\varepsilon\left\Vert \left(  I-\pi_{M}\right)  \left(  \varepsilon\left(
I-\pi_{M}\right)  +\Gamma_{0}^{T}\left(  \widetilde{F}\right)  \right)
^{-1}\left(  \widetilde{h}\right)  \right\Vert .\nonumber
\end{gather}
By (\ref{ww3}) and the uniform boundedness of $\left\Vert \varepsilon\left(
\varepsilon\left(  I-\pi_{M}\right)  +\Gamma_{0}^{T}\left(  F\left(
y_{\varepsilon}^{\ast}\right)  \right)  \right)  ^{-1}\right\Vert ,$ the first
term goes to zero. The second term approaches zero thanks to Lemma \ref{lem5}.
The last converges to zero according to Theorem \ref{thm:4}. Thus taking limit
in (\ref{ww11}) we complete the proof.
\end{proof}

\begin{remark}
If in (\ref{mp1}) $f=0$ we get the finite-approximate controllability of
semilinear system with bounded nonlinear term (cf \cite{mah3}). If $g=0$ we
get the finite-approximate controllability of semilinear system with the
Lipschitz nonlinear term. So the result is new even for the case $g=0$.
\end{remark}

\section{ Applications}

\textit{Example 1. } Consider the partial differential system of the form%
\begin{align}
\dfrac{\partial}{\partial t}y\left(  t,\theta\right)   &  =\dfrac{\partial
^{2}}{\partial\theta^{2}}y\left(  t,\theta\right)  +m\left(  \theta\right)
u\left(  t,\theta\right)  +f\left(  y\left(  t,\theta\right)  \right)
+g\left(  y\left(  t,\theta\right)  \right)  ,\left(  t,\theta\right)
\in\left(  0,T\right)  \times\left(  0,\pi\right) \nonumber\\
y\left(  t,\theta\right)   &  =0,\;\left(  0,T\right)  \times\left\{
0,\pi\right\}  ,\label{ap}\\
y\left(  0,\theta\right)   &  =y_{0}\left(  \theta\right)  ,\ \theta\in\left[
0,\pi\right]  ,\ \ y_{0}\in L^{2}\left[  0,\pi\right]  ,\nonumber
\end{align}
$m$ is the characteristic function of an open subset $\omega\subset\left[
0,\pi\right]  .$ We assume that $f\in C^{1}(R)$ and $\left\vert f^{\prime
}\left(  r\right)  \right\vert \leq L$ for all $r\in R.$ So $f$ is globally
Lipschitz. Moreover assume that $g\in C(R)$ and $\left\vert g\left(  r\right)
\right\vert \leq M$ for all $r\in R.$

To write the system (\ref{ap}) in a semigroup form define $\mathfrak{X}%
=U=L^{2}\left[  0,\pi\right]  $ and $A:D\left(  A\right)  \subset
\mathfrak{X}\rightarrow\mathfrak{X}$ to be $Ay=y^{\prime\prime},$ where
$D\left(  A\right)  =H_{0}^{1}\left[  0,\pi\right]  \cap H^{2}\left[
0,\pi\right]  .$ We also define the operators $F,G:X\rightarrow X$ by $\left(
Fy\right)  \left(  \theta\right)  =f\left(  y\left(  \theta\right)  \right)
,\ $ $\left(  Gy\right)  \left(  \theta\right)  =g\left(  y\left(
\theta\right)  \right)  $for almost every $\theta\in\left[  0,\pi\right]  $
and the bounded linear control operator $B:\mathfrak{X}\rightarrow
\mathfrak{X}$ by $\left(  Bu\right)  \left(  \theta\right)  =m\left(
\theta\right)  u\left(  t,\theta\right)  $ for almost every $\theta\in\left[
0,\pi\right]  .$

Taking into account all these notations, the state system (\ref{ap}) becomes%
\begin{align*}
y^{\prime}+Ay  &  =Bu+F\left(  y\right)  +G\left(  y\right)  \ \ \text{on
}\left(  0,T\right) \\
y\left(  0\right)   &  =y_{0}.
\end{align*}
We know that $A$ generates a compact $C_{0}$-semigroup, $F$ is globally
Lipschitz on $\mathfrak{X}$ and $B$ is bounded. So, for each $u\in
L^{2}\left(  0,T;\mathfrak{X}\right)  $ and $y_{0}\in\mathfrak{X},$ $\left(
\ref{ap}\right)  $ has a unique mild solution $y\in C\left(  0,T;X\right)  $.

It is known that the linearized system associated with (\ref{ap}) is
finite-approximately controllable, see \cite{zuazua1}. Thus by Theorem
\ref{thm1} the system (\ref{ap}) is finite-approximately controllable on
$\left[  0,T\right]  $.

\textit{Example 2}. We consider a system governed by the semilinear heat
equation with lumped control%
\begin{equation}
\left\{
\begin{array}
[c]{c}%
\dfrac{\partial x\left(  t,\theta\right)  }{\partial t}=\dfrac{\partial
^{2}x\left(  t,\theta\right)  }{\partial\theta^{2}}+\chi_{\left(  \alpha
_{1},\alpha_{2}\right)  }\left(  \theta\right)  u\left(  t\right)  +g\left(
x\left(  t,\theta\right)  \right)  ,\\
x\left(  t,0\right)  =x\left(  t,\pi\right)  =0,\ \ \ \ 0<t<T,\\
x\left(  0,\theta\right)  =x_{0}\left(  \theta\right)  ,\ \ \ \ 0\leq
\theta\leq\pi,
\end{array}
\right.  \label{exp1}%
\end{equation}
where $\chi_{\left(  \alpha_{1},\alpha_{2}\right)  }\left(  \theta\right)  $
is the characteristic function of $\left(  \alpha_{1},\alpha_{2}\right)
\subset\left(  0,\pi\right)  $. Let $\mathfrak{X}=L^{2}\left[  0,\pi\right]
$, $U=R$, and $A=d^{2}/d\theta^{2}$ with $D\left(  A\right)  =H_{0}^{1}\left[
0,\pi\right]  \cap H^{2}\left[  0,\pi\right]  .$ We define the bounded linear
operator $B:R\rightarrow L^{2}\left[  0,\pi\right]  $ by $(Bu)(t)=\chi
_{\left(  \alpha_{1},\alpha_{2}\right)  }\left(  \theta\right)  u\left(
t\right)  $, and the nonlinear operator $g$ is assumed to satisfy (G).

If $\alpha_{1}\pm\alpha_{2}$ is an irrational number, then the linear system
corresponding to (\ref{exp1}) is finite-approximately controllable (see
\cite{khapalov}), and by Theorem \ref{thm1}, the system (\ref{exp1}) is
finite-approximately controllable on $\left[  0,T\right]  $, provided that
condition (G) is satisfied.

\textit{Example 3.} Consider the following initial-boundary value problem of
parabolic control system
\begin{equation}
\left\{
\begin{tabular}
[c]{lll}%
$\dfrac{\partial}{\partial t}y\left(  t,\theta\right)  =\dfrac{\partial^{2}%
}{\partial\theta^{2}}y\left(  t,\theta\right)  +Bu\left(  t,\theta\right)
+g\left(  t,y\left(  t,\theta\right)  \right)  ,$ & $t\in\left[  0,1\right]
,\theta\in\left[  0,\pi\right]  ,$ & \\
$y\left(  t,0\right)  =y\left(  t,\pi\right)  =0,$ & $t\in\left[  0,1\right]
,$ & \\
$y\left(  0,\theta\right)  =y_{0}\left(  \theta\right)  ,$ & $t\in\left[
0,1\right]  ,\theta\in\left[  0,\pi\right]  .$ &
\end{tabular}
\ \ \ \right.  \label{ex2}%
\end{equation}
Take $\mathfrak{X}=U=L^{2}\left[  0,\pi\right]  $ and the operator $A:D\left(
A\right)  \subset\mathfrak{X}\rightarrow\mathfrak{X}$ is defined as in
\textit{Example 1}. Define the operator $B$ as follows%
\[
Bu\left(  t\right)  =\sum_{n=1}^{\infty}\overline{u}_{n}\left(  t\right)
e_{n},
\]
where%
\begin{align*}
u\left(  t\right)   &  =\sum_{n=1}^{\infty}\left\langle u\left(  t\right)
,e_{n}\right\rangle e_{n},\\
\overline{u}_{n}\left(  t\right)   &  =\left\{
\begin{tabular}
[c]{ll}%
$0,$ & $0\leq t<1-\dfrac{1}{n^{2}},$\\
$\left\langle u\left(  t\right)  ,e_{n}\right\rangle ,$ & $1-\dfrac{1}{n^{2}%
}\leq t\leq1,$%
\end{tabular}
\ \ \ \right.
\end{align*}
then, one can easily obtain that $\left\Vert Bu\right\Vert \leq\left\Vert
u\right\Vert ,$ which implies that $B$ is bounded. It is known that the linear
system corresponding to (\ref{ex2}) is approximately controllable. By Theorem
\ref{thm1}, the system (\ref{ex2}) is finite-approximately controllable on
$\left[  0,T\right]  $, provided that condition (G) is satisfied.

\end{document}